\documentclass[a4paper,12pt,reqno]{amsart}

\usepackage{amsfonts}
\usepackage{amsmath}
\usepackage{amssymb}

\usepackage{mathrsfs}

\usepackage[colorlinks]{hyperref}
\renewcommand\eqref[1]{(\ref{#1})}
\numberwithin{equation}{section}

\usepackage{amssymb,amsmath,amsthm}
\usepackage{graphicx}
\usepackage{xcolor}

\usepackage{amsfonts}
\usepackage{latexsym}
\usepackage[english]{babel}

\newcommand{\Rn}{{\mathbb R}^n}

\numberwithin{equation}{section}
\newtheorem{theorem}{Theorem}[section]
\newtheorem{proposition}[theorem]{Proposition}
\newtheorem{lemma}[theorem]{Lemma}
\newtheorem{corollary}[theorem]{Corollary}

\newtheorem{remark}[theorem]{Remark}
\newtheorem{remarks}[theorem]{Remark}
\newtheorem{definition}[theorem]{Definition}
\newcommand{\be}{\begin{equation}}
\newcommand{\ee}{\end{equation}}

\newcommand{\R}{\mathbb R}
\newcommand{\C}{\mathbb C}

\newcommand{\N}{\mathbb N}

\newcommand{\ii }{{\rm i} }

\newcommand{\Opt}{{\rm Op}_\tau}
\newcommand{\Op}{{\rm Op}}

\begin{document}

\title[$\tau$-quantizations of pseudo-differential operators]
{Pseudo-differential operators with nonlinear quantizing functions}

\author[Massimiliano Esposito]{Massimiliano Esposito}
\address{
  Massimiliano Esposito:
   \endgraf
    Department of Mathematics
  \endgraf
  Imperial College London
  \endgraf
  180 Queen's Gate, London, SW7 2AZ
  \endgraf
  United Kingdom
  \endgraf
  {\it E-mail address} {\rm m.esposito14@imperial.ac.uk}
 }

\author[Michael Ruzhansky]{Michael Ruzhansky}
\address{
  Michael Ruzhansky:
  \endgraf
  Department of Mathematics
  \endgraf
  Imperial College London
  \endgraf
  180 Queen's Gate, London, SW7 2AZ
  \endgraf
  United Kingdom
  \endgraf
  {\it E-mail address} {\rm m.ruzhansky@imperial.ac.uk}
  }

\thanks{The second author was supported by the EPSRC
grant EP/R003025/1 and by the Leverhulme Grant RPG-2017-151. No new data was collected or generated during the course of research.}

\date{\today}


\begin{abstract}
In this paper we develop the calculus of pseudo-differential operators corresponding to the quantizations of the form
$$
Au(x)=\int_{\mathbb{R}^n}\int_{\mathbb{R}^n}e^{\ii(x-y)\cdot\xi}\sigma(x+\tau(y-x),\xi)u(y)dyd\xi,
$$
where $\tau:\Rn\to\Rn$ is a general function. In particular, for the linear choices $\tau(x)=0$, $\tau(x)=x$, and $\tau(x)=\frac{x}{2}$ this covers the well-known Kohn-Nirenberg, anti-Kohn-Nirenberg, and Weyl quantizations, respectively. Quantizations of such type appear naturally in the analysis on nilpotent Lie groups for polynomial functions $\tau$ and here we investigate the corresponding calculus in the model case of $\Rn$. We also give examples of nonlinear $\tau$ appearing on the polarised and non-polarised Heisenberg groups.
\end{abstract}

\maketitle

\tableofcontents

\section{Introduction}

There are many ways (quantizations) of associating the operator to a function of variables $(x,\xi)$ on the phase space. In this paper we are going to discuss a generalisation of the well-known quantization procedures:
\begin{eqnarray*}
\mbox{Weyl}\qquad\frac{1}{(2\pi)^n}\int_{\mathbb{R}^n}\int_{\mathbb{R}^n}e^{\ii(x-y)\cdot\xi}\sigma(\frac{x+y}{2},\xi)u(y)dyd\xi,\\
\mbox{Kohn-Nirenberg}\qquad\frac{1}{(2\pi)^n}\int_{\mathbb{R}^n}\int_{\mathbb{R}^n}e^{\ii(x-y)\cdot\xi}\sigma(x,\xi)u(y)dyd\xi,\\
\mbox{Anti-Kohn-Nirenberg}\qquad\frac{1}{(2\pi)^n}\int_{\mathbb{R}^n}\int_{\mathbb{R}^n}e^{\ii(x-y)\cdot\xi}\sigma(y,\xi)u(y)dyd\xi.
\end{eqnarray*}
More specifically, we are going to study pseudo-differential operators of the form
\begin{eqnarray}\label{psuedo}
\frac{1}{(2\pi)^n}\int_{\mathbb{R}^n}\int_{\mathbb{R}^n}e^{\ii(x-y)\cdot\xi}\sigma(x+\tau(y-x),\xi)u(y)dyd\xi,
\end{eqnarray}
where $\tau$ is a function satisfying certain admissibility conditions. In particular, we can recover the above listed quantizations by taking $\tau(x)=\frac{x}{2}$ (Weyl), $\tau(x)=0$ (Kohn-Nirenberg), and $\tau(x)=x$ (anti-Kohn-Nirenberg). 
In the more general special case of a linear function $\tau(x)=s x$ with $0\leq s\leq 1$ we recover the quantizations analysed in detail by 
Shubin \cite{Schubin}. The quantizations with linear $\tau$ of the form $\tau(x)=Mx$ for some non-degenerate matrix $M$ have been also considered recently in the context of modulation \cite{Toft} or \cite{Bayer}, and Gelfand-Shilov spaces \cite{CT}.

In this paper we will consider two cases of $\tau$, with unbounded and with bounded derivatives, and also symbols satisfying inequalities of the form
\begin{equation}\label{STIMESIGMATAU}
|\partial_{x}^{\alpha}\partial_{y}^{\beta}\partial_{\xi}^{\gamma}\left[\sigma(x+\tau(y-x),\xi)\right]|\le C_{\alpha,\beta,\gamma}\langle\xi\rangle^{m-|\gamma|}\langle x-y\rangle^{d(|\alpha|+\beta|)},
\end{equation}
or of the form
\begin{equation}\label{STIMESIGMATAU2}
|\partial_{x}^{\alpha}\partial_{y}^{\beta}\partial_{\xi}^{\gamma}\left[\sigma(x+\tau(y-x),\xi)\right]|\le C_{\alpha,\gamma}\langle\xi\rangle^{m-|\gamma|},
\end{equation}
where $\alpha,\beta,\gamma\in\mathbb{N}^{n}$ are multi-indices, $m,d\in\mathbb{R}$ are numbers depending on the symbol $\sigma$ and on $\tau$, and $\langle\xi\rangle=(1+|\xi|^2)^{1/2}.$

We are going to develop a calculus for \eqref{psuedo}, i.e. we prove the adjoint, composition, and other formulae, make links between quantizations for different choices of $\tau$, and investigate different properties of operators of this kind. Moreover, we discuss the Calder\'on-Vaillancourt theorem, ellipticity and parametrix, and the G{\aa}rding inequality.

We may skip the constant $\frac{1}{(2\pi)^n}$ on most occasions but we have put it in \eqref{psuedo} to ensure that the constant $1$ is quantized into the identity operator.

The quantizations of the form \eqref{psuedo} appear naturally in the analysis on nilpotent Lie groups, in particular in the question of finding suitable analogues of the Weyl quantizations in the noncommutative setting. In \cite{MR}, a class of so-called symmetric quantizations was identified inheriting the important property for quantum physics, that the self-adjoint symbols are quantized into self-adjoint operators. When written in local coordinates, this reduces to quantizing the symbol in the form $\sigma(x+\tau(y-x),\xi)$ for suitable choices of non-linear functions $\tau$. 

For example, if $G$ is a locally compact unimodular group of type I, and $\tau:G\rightarrow G$ is a measurable function, general $\tau$-quantizations on $G$ were considered in \cite{MR} in the form
\begin{equation}\label{bilfred}
{\rm Op}^\tau(\sigma)u(x)=\int_G\!\Big(\int_{\widehat{G}}\,{\rm Tr}_\xi\Big[\xi(y^{-1}x)\sigma\big(x\tau(y^{-1}x)^{-1}\!,\xi\big)\Big]d\mu(\xi)\Big)u(y)dm(y),
\end{equation}
where $dm$ and $d\mu$ are the Haar and the Plancherel measures on $G$ and $\widehat{G}$, respectively. We refer to \cite{MR} for the details, but here we can say that 
it was shown in \cite[Proposition 4.3]{MR} that
${\rm Op}^\tau(\sigma)^*={\rm Op}^\tau(\sigma^*)$ if and only if $\tau(x)=\tau(x^{-1})x$ for all $x\in G$, and functions $\tau$ satisfying this condition were called symmetry functions. Moreover, it was also shown that if $G$ is an exponential group (i.e. the exponential mapping is a global diffeomorphism) then such symmetry functions always exist, e.g. given by
\begin{equation}\label{EQ:Mrsym}
\tau(x):=\int_{0}^{1}\exp[s\log x]ds.
\end{equation} 
In the case of $G=\Rn$ the quantization \eqref{bilfred} reduces to \eqref{psuedo}, and since both mappings $\exp$ and $\log$ are identities, formula \eqref{EQ:Mrsym} boils down (modulo signs) to $\tau(x)=\frac12 x$, yielding the usual Weyl quantization, so that real-valued (self-adjoint) symbols are quantized into self-adjoint operators. However, already on nilpotent Lie groups different from $\Rn$, since the group law is polynomial, the symmetry functions $\tau$ in \eqref{EQ:Mrsym} do not have to be linear. 

We give a further example for this construction in the setting of the Heisenberg group in Section \ref{SEC:heisenberg}, in particular, the symmetry function \eqref{EQ:Mrsym} on the Heisenberg group $\mathbb H\simeq \mathbb R^3$ is already nonlinear, taking the form
\eqref{EQ:tau-heis2}, namely,
\begin{equation}\label{EQ:tauheisex}
\tau(a,b,c)=\left(\frac{a}{2},\frac{b}{2},\frac{c}{2}+\frac{ab}{6}\right).
\end{equation} 
Consequently, we get the formula for the `midpoint' function $m(x,y)=x\tau(y^{-1}x)^{-1}$
from \eqref{EQ:midpoint} in the Weyl-type quantization \eqref{bilfred} as
$$
m((a_1,b_1,c_1),(a_2,b_2,c_2)) =\left(\frac{a_1+a_2}{2},\frac{b_1+b_2}{2},\frac{c_1+c_2}{2}
-\frac{(a_1-a_2)(b_1-b_2)}{6}
\right).
$$
Curiously, such a midpoint between $x$ and $x^{-1}$ is not the origin but
$$
m((a,b,c),(a,b,c)^{-1}) =\left(0,0,
-\frac{2ab}{3}
\right).
$$

Since such quantizations appear not studied even in the simplest settings, we took it as an incentive to analyse it in this paper first in the simplest classical setting of $\Rn$. This already yields some insights and further intuition into such quantizations.

\smallskip
Throughout this paper symbols will be mainly denoted with $\sigma$, while their associated $\tau$-quantization will be marked with a subscript $\tau$. Also, instead of writing inequalities like $|\partial_{x}^{\alpha}\partial_{\xi}^{\beta}\sigma(x,\xi)|\le C_{\alpha}\langle\xi\rangle^{m-|\beta|}$ we will often write $|\partial_{x}^{\alpha}\partial_{\xi}^{\beta}\sigma(x,\xi)|\lesssim_{\alpha} \langle\xi\rangle^{m-|\beta|}$ to denote a possible dependence on a constant $C_{\alpha}$.
Similarly, we will simply write $A\lesssim B$ if there is a constant $C>0$ such that $A\leq CB.$
For clarify, we will be also using the notation $\mathbb{N}_0=\mathbb N\cup\{0\}.$

\smallskip
The authors would like to thank Julio Delgado for discussions and for comments on the preliminary version of the manuscript.

\section{Admissible $\tau$-quantizations}
\label{SEC:2}

In this section we are going to introduce an admissible class of functions $\tau$ and the corresponding classes of symbols and amplitudes. 

\begin{definition}\label{TAO AMMISSIBILE}
A function $\tau:\mathbb{R}^{n}\rightarrow\mathbb{R}^{n}$ will be called admissible if $\tau\in C^{\infty}$ and $\tau(0)=0$, and if there exists $\mu\geq 0$ such that
\begin{equation}\label{EQ:admiss-cond}
\\|\partial_{x}^{\alpha}\tau(x)|\lesssim_{\alpha}\langle x\rangle^{\mu}  
\end{equation}
holds for all multi-indices $\alpha\in\mathbb{N}_0^{n}\backslash\{0\}$ and all $x\in\Rn.$ 
We will often say that in this case $\tau$ is admissible of order $\mu$. We will also often call $\tau$ the quantizing function.
\end{definition}

We are now going to introduce the class of symbols and amplitudes that we will use and the associated operators. 
Given an admissible $\tau$, we say that a smooth function $\sigma:\mathbb{R}^n\times \R^n  \rightarrow \C$ belongs to the class $S^{m}_{d,\tau}$, where $m$ and $d$ are real numbers, if for all multi-indices $\alpha$,$\beta$,$\gamma\geq 0$ and all $(x,y, \xi) \in \R^n\times\R^n \times \R^n$ we have
\begin{equation}\label{EQ:s-class}
|\partial_x^\alpha\partial_{y}^{\beta} \partial_\xi^\gamma 
\left[\sigma(x+\tau(y-x), \xi)\right]| \lesssim_{\alpha, \beta,\gamma} \langle  \xi \rangle^{m - |\gamma|}\langle x-y\rangle^{d\cdot(|\alpha|+|\beta|)}.
\end{equation} 
A typical situation would be to take $\sigma\in S^m$, that is, satisfying
$$
|\partial_x^\alpha\partial_\xi^\gamma 
\sigma(x, \xi)| \lesssim_{\alpha, \gamma} \langle  \xi \rangle^{m - |\gamma|}.
$$
Consequently, if $\tau$ is an admissible quantizing function of order $\mu$, then $\sigma\in S^m_{\mu,\tau}.$

We note that the derivatives in $y$ can be eliminated from the assumption \eqref{EQ:s-class}, namely, the class $S^{m}_{d,\tau}$ can be also characterised by the condition that for all multi-indices $\alpha, \gamma\geq 0$ and all $(x,y, \xi) \in \R^n\times\R^n \times \R^n$ we have
\begin{equation}\label{EQ:s-class2}
|\partial_x^\alpha \partial_\xi^\gamma 
\left[\sigma(x+\tau(y-x), \xi)\right]| \lesssim_{\alpha, \gamma} \langle  \xi \rangle^{m - |\gamma|}\langle x-y\rangle^{d|\alpha|}.
\end{equation}

Given $\sigma \in S^m_{d,\tau}$ and $u\in\mathcal{S}(\R^n)$ we define the operator $A_{\sigma,{\tau}} \equiv \Opt(\sigma)$ associated to $\sigma$ as
\begin{equation}\label{EQ:quant}
A_{\sigma,{\tau}} u(x) = \Opt(\sigma) u(x) := \int_{\R^n} \int_{\R^n} e^{\ii (x - y) \cdot \xi} \sigma(x+\tau(y-x), \xi) u(y) \,d y \, d \xi. 
\end{equation}
In this case we also say that $A_{\sigma,{\tau}}\in OP_\tau S^{m}_{d,\tau}.$
For $\tau(x)=0$ we have the usual Kohn-Nirenberg quantization, and we can abbreviate this case by writing $\Op=\Op_{\tau\equiv 0}.$

We also define the corresponding class of amplitudes $a(x, y, \xi)$.
Namely, a smooth function $a : \R^n \times \R^n \times \R^n \to \C$ belongs to the class $A^{m}_{d}$, where $m$ and $d$ are real numbers if for any $\alpha, \beta, \gamma \geq 0$ and for all  $(x, y, \xi) \in \R^n \times \R^n \times \R^n$ one has the inequalities
\begin{equation}\label{EQ:s-amp}
|\partial_x^\alpha \partial_y^\beta \partial_\xi^\gamma a(x, y, \xi)| \lesssim_{\alpha, \beta, \gamma} \langle \xi \rangle^{m - |\gamma|}\langle x-y\rangle^{d(|\alpha|+|\beta|)} .
\end{equation} 
Clearly the classes $S^{m}_{d,\tau}$  depend on $\tau$, but for simplicity we may sometimes suppress the letter $\tau$ from their notation as it should cause no confusion when $\tau$ is fixed.

Given $a \in A^m_d$ and $u\in\mathcal{S}({\mathbb{R}^n})$, we denote the operator  $A_{a} := \Op(a) $ associated to $a$ as
\begin{equation}\label{EQ:quant-amp}
A_{a} u(x) = \Op(a) u(x) :=  \int_{\R^n} \int_{\R^n} e^{\ii (x - y) \cdot \xi} a(x, y, \xi) u(y) \,d y \, d \xi.
\end{equation} 
In this case we also write $A_a\in OP A^m_d$.

\section{Calculus associated to $\tau$-quantization }
\label{SEC:3}

In this section we are going to show the calculus properties of general operators of the forms \eqref{EQ:quant} and \eqref{EQ:quant-amp}.
In particular we will prove that any operator of the form \eqref{EQ:quant-amp} is almost equivalent  to an operator of the form \eqref{EQ:quant}.
For this we will need the following preparatory basic statements.

We say that a function $a : \R^n \times \R^n \times \R^n \to \C$ belongs to the class $A^{m,k}$, where $m\in\R$ and $k\in\mathbb N$ if for all  $(x, y, \xi) \in \R^n \times \R^n \times \R^n$ one has the inequalities
\begin{equation}\label{EQ:s-amp-k}
|\partial_x^\alpha \partial_y^\beta \partial_\xi^\gamma a(x, y, \xi)| \lesssim_{\alpha, \beta, \gamma} \langle \xi \rangle^{m - |\gamma|}
\end{equation} 
for any multi-indices $|\alpha|, |\beta|\leq k$ and any $\gamma$.

\begin{lemma}\label{EQUIVALENZA}
If an amplitude operator $A_a$ is associated to the amplitude $a\in A^{m}_{d}$ then for any $k\in\mathbb N$ there exists $a_k\in A^{m,k}$ such that $A_a=A_{a_k}$.
\end{lemma}
If the above property holds we will sometimes say that $A_a$ {\em weakly belongs} to $A^m$, or that $A_a$ is weakly equivalent to an operator with an amplitude in $A^m$. 

For many questions this is actually as good as the strong identification, especially for obtaining results depending only on a finite number of derivatives of symbols (such as e.g. Calder\'on-Vaillancourt theorem, see Theorem \ref{THM:CV}, and many other results in the theory of pseudo-differential operators).

\proof
Let us denote
$$\mathcal{L}:=\frac{1-\Delta_{\xi}}{1+|x-y|^2},$$
where $\Delta_\xi$ is the usual Laplacian on $\Rn$. For any $N\in\mathbb N$, applying $\mathcal{L}^{N}$ to $e^{\ii(x-y)\cdot\xi}$ we have
$\mathcal{L}^{N}e^{\ii(x-y)\cdot\xi}=e^{\ii(x-y)\cdot\xi}.$
Substituting this inside $A_a$ and integrating by parts we get
\begin{align*}
A_af(x) & =\int_{\mathbb{R}^{n}}\int_{\mathbb{R}^{n}}e^{\ii(x-y)\cdot\xi}a(x,y,\xi)f(y)dyd\xi\\
&=\int_{\mathbb{R}^{n}}\int_{\mathbb{R}^{n}}\mathcal{L}^{N}e^{\ii(x-y)\cdot\xi}a(x,y,\xi)f(y)dyd\xi\\
&=\int_{\mathbb{R}^{n}}\int_{\mathbb{R}^{n}}e^{\ii(x-y)\cdot\xi}\frac{(1-\Delta_{\xi})^{N}}{(1+|x-y|^2)^N}a(x,y,\xi)f(y)dyd\xi.
\end{align*}
Then we can take $a_k(x,y,\xi):=\frac{(1-\Delta_{\xi})^{N}}{(1+|x-y|^2)^N}a(x,y,\xi)$ for 
any $N>\mu k$.
\endproof

Let us also record the following property which will be of use later. As usual, for $w\in\Rn$ and a multi-index $\gamma$ we will be using the notation $w^\gamma=w_1^{\gamma_1}\cdots w_n^{\gamma_n}.$

\begin{lemma}\label{PRODPOLINOMI}
Given a function $\tau :\mathbb{R}^{n}\rightarrow\mathbb{R}^{n}$ of the form
\begin{equation}\label{polinomio}
\tau(w)=\sum_{1\le|\gamma|\le N-1} c_{\gamma}(\tau)w^{\gamma}+\sum_{|\gamma|=N}c_{\gamma}(\tau,w)w^{\gamma}\end{equation} 
with an integer $N\geq 2$,
\begin{equation*}c_{\gamma}(\tau)=(c^1_{\gamma}(\tau),c^2_{\gamma}(\tau),\dots,c^n_{\gamma}(\tau))\in\mathbb{R}^{n}
\end{equation*}
and
\begin{equation*}c_{\gamma}(\tau,w)=(c^1_{\gamma}(\tau,w),c^2_{\gamma}(\tau,w),\dots,c^n_{\gamma}(\tau,w)),
\end{equation*}
for all multi-indices $\alpha$, $\beta\in\mathbb{N}_0^{d}$ we have
\begin{equation}\label{polinomio da usare}
[\tau(w)]^{\alpha}[w-\tau(w)]^{\beta}=\sum_{|\alpha|+|\beta|\le|\delta|\le N\cdot{(|\alpha|+|\beta|)}}E_\delta(\tau,w)\prod_{i=1}^{n}w_i^{\delta_i}
\end{equation}
where each $E_\delta(\tau,w)$ is either a constant or the product of at most $|\alpha|+|\beta|$ constants $c_\gamma(\tau)$ or functions $c_\gamma(\tau,w)$.
In particular, we have $E_0(\tau,w)\equiv 1.$
\end{lemma}
The statement follows by observing that the left hand side of \eqref{polinomio da usare} has zero at $w=0$ of order at least $|\alpha|+|\beta|.$

We will now show that any amplitude operator \eqref{EQ:quant-amp} with the amplitude in $ A^{m}_{d}$ can be written as a $\tau$-quantization of some symbol. 

\begin{theorem} 
\label{teorema da amplitude a simboli}
Let $a \in A^{m}_{d}$ and let us consider the associated operator $A_a$ defined by 
$$
A_a u(x) := \int_{\R^n} \int_{\R^n} e^{\ii (x - y) \cdot \xi} a(x, y, \xi) u(y) \,d y \, d \xi, \qquad u \in {\mathcal S}(\R^n). 
$$
Then, for any admissible function $\tau$ of order $\mu\geq 0$, there exists a symbol $\sigma$ 
such that 
$$
A_a u(x) = \int_{\R^n} \int_{\R^n} e^{\ii (x - y) \cdot \xi} \sigma(x + \tau(y - x), \xi) u(y) \, d y\, d \xi,
$$
and $\sigma$ has the following {\em weak} asymptotic expansion
\begin{multline}\label{EQ:wasexp}
\sigma(x+\tau(y-x),\xi)\sim
\\ \sum_{\alpha,\beta\geq 0}\sum_{|\alpha| + |\beta|\le |\delta |\le N(|\alpha| + |\beta|)} k_\delta(\tau,\alpha,\beta,x-y)\partial_{x}^{\alpha}\partial_{y}^{\beta}\partial_{\xi}^{\delta}a(v,v,\xi)|_{v=x+\tau(y-x)},
\end{multline} 
where we can take any $N\geq 1$, and where each of the terms 
$$k_\delta(\tau,\alpha,\beta,x-y)\partial_{x}^{\alpha}\partial_{y}^{\beta}\partial_{\xi}^{\delta}a(v,v,\xi)|_{v=x+\tau(y-x)}$$ is a symbol in $S^{m-|\delta|}_{\mu(|\alpha|+|\beta|),\tau}$. Moreover, we have $k_0(\tau,0,0,x-y)\equiv1$.

The weak asymptotic expansion, in \eqref{EQ:wasexp} and in the sequel, will mean that it is an asymptotic expansion in $\xi$ in the classical sense locally in space variables, and globally, for any $M\in\mathbb N_0$, we have
\begin{multline}\label{EQ:aexpm}
\sigma(x+\tau(y-x),\xi)-\sum_{|\alpha|+|\beta|\leq M}\sum_{|\alpha| + |\beta|\le |\delta |\le N(|\alpha| + |\beta|)} k_\delta(\tau,\alpha,\beta,x-y) \times \\
\times \partial_{x}^{\alpha}\partial_{y}^{\beta}\partial_{\xi}^{\delta}a(v,v,\xi)|_{v=x+\tau(y-x)}\in A^{m-M}_{(\mu+d)M}.
\end{multline} 
In such cases we will also say that $\sigma$ is an asymptotic sum with terms in $S^{m-(|\alpha|+|\beta|)}_{\mu(|\alpha|+|\beta|),\tau}$ and remainder in $A^{m-(|\alpha|+|\beta|)}_{(\mu+d)(|\alpha|+|\beta|)}.$

By Lemma \ref{EQUIVALENZA} each of the symbols $k_\delta(\tau,\alpha,\beta,x-y)\partial_{x}^{\alpha}\partial_{y}^{\beta}\partial_{\xi}^{\delta}a(v,v,\xi)|_{v=x+\tau(y-x)}\in S^{m-|\delta|}_{\mu(|\alpha|+|\beta|),\tau}$ gives rise to a pseudo-differential operator with an amplitude weakly belonging to $A^{m-|\delta|}_0$.
\end{theorem}

\begin{remark}\label{REM:NN}
We note that the number $N$ in \eqref{EQ:wasexp} can be any integer number. It is related to how many terms in the Taylor expansion of $\tau$ we will take in the proof in formula \eqref{EQ:tauexp}. The simplest formulae are obtained by taking $N=1$, in which case all the asymptotic expansions in this paper are simplified since then the sum $\sum_{|\alpha| + |\beta|\le |\delta |\le N(|\alpha| + |\beta|)}$ is just a simple sum $\sum_{|\delta |=|\alpha| + |\beta|}$. However, 
for more flexibility we allow for any $N$ in the expressions allowing for a slightly more general form of the expansions. These are useful in the case when $\tau$ is polynomial, as e.g. in \eqref{EQ:tauheisex}.
\end{remark}

\begin{proof}[Proof of Theorem \ref{teorema da amplitude a simboli}]
Let us consider the changes of variables 
$$
\begin{cases}
v:=x + \tau(y - x), \\
w:=x - y,
\end{cases}
$$
so that we have
$$
\begin{cases}
x = v - \tau(- w), \\
y = v - w - \tau(- w). 
\end{cases}
$$
Then in these new variables we can write 
$$
a (x, y, \xi) = a (v - \tau(- w), v - w - \tau(- w), \xi). 
$$
We now expand $a(x,y,\xi)$ in the Taylor series in the first two variables around the point $(v, v, \xi)$ obtaining, for any $M \in \N$, that 
\begin{align}
a (x, y, \xi) & = \sum_{|\alpha| + |\beta| < M} \frac{\partial_x^{\alpha} \partial_y^\beta a(v, v, \xi) [-\tau(- w)]^\alpha [ - w - \tau(- w)]^\beta }{\alpha ! \beta !}+ r_M(x, y, \xi)\,, \label{taylor simbolo a con tau alpha beta}
\end{align} 
with the remainder given by 
\begin{multline}\label{definizione resto di Taylor teo Schubin}
r_M(x, y, \xi) = \frac{1}{(M - 1)!}\sum_{|\alpha| + |\beta| = M} [-\tau(- w)]^\alpha [ - w - \tau(- w)]^\beta \times \\ \times \int_0^1 (1 - t)^{M - 1} \partial_x^{\alpha} \partial_y^\beta a\Big(v - t \tau(- w) , v + t( - w - \tau(- w)), \xi \Big)\, d t. 
\end{multline}
To analyse the terms $[\tau(- w)]^\alpha [ - w - \tau(- w)]^\beta$,
we do another Taylor expansion of $\tau$ around 0  writing it in the form \eqref{polinomio}, that is, for any $N\geq 1$, 
\begin{equation}\label{EQ:tauexp}
\tau(w)=\sum_{1\le|\alpha_0|\le N-1}c_{\alpha_0}(\tau)w^{\alpha_0}+\sum_{|\beta_0|=N}c_{\beta_0}(\tau,w)w^{\beta_0}.
\end{equation} 
By Lemma \ref{PRODPOLINOMI} we can write is then in the form
 \begin{equation}\label{polinomi alpha beta tau primo teorema}
[\tau(- w)]^\alpha [ - w - \tau(- w)]^\beta = \sum_{|\alpha|+|\beta|\le|\delta | \le N\cdot(|\alpha| + |\beta|)} E_\delta(\tau,w) w^\delta\,,
\end{equation}
for some $E_\delta$ up to the signs coinciding with that in \eqref{polinomio da usare}.
Now recalling that $w = x - y$,  for any $|\alpha| + |\beta| < M$ we have
\begin{multline*}
\partial_x^{\alpha} \partial_y^\beta a(v, v, \xi) [\tau(- w)]^\alpha [ - w - \tau(- w)]^\beta  \\
= \sum_{|\alpha|+|\beta|\le|\delta | \le N\cdot(|\alpha| + |\beta|)} E_\delta(\tau,x-y) \partial_x^{\alpha} \partial_y^\beta a(v, v, \xi) (x - y)^\delta.
\end{multline*} 
We note that integrating by parts in $\xi$ under the quantization integral, the operator associated to the amplitude $E_{\delta}(\tau,x-y)\partial_x^{\alpha} \partial_y^\beta a(v, v, \xi)(x-y)^{\delta}$ is the same as the one associated to the amplitude $E_{\delta}(\tau,x-y)(\frac{-1}{\ii})^{|\delta|} \partial_x^{\alpha} \partial_y^\beta \partial_\xi^\delta a(v, v, \xi) $. Consequently, we obtain the equality 
\begin{multline*}
{\rm Op}\Big( \partial_x^{\alpha} \partial_y^\beta a(v, v, \xi) [\tau(- w)]^\alpha [ - w - \tau(- w)]^\beta \Big) \\
= \sum_{|\alpha|+|\beta|\le|\delta | \le N\cdot(|\alpha| + |\beta|)}  {\rm Op}\Big(  \ii^{|\delta|} E_\delta(\tau,x-y)\partial_x^{\alpha} \partial_y^\beta \partial_\xi^\delta a(v, v, \xi) \Big).
\end{multline*} 
According to \eqref{taylor simbolo a con tau alpha beta}, defining
\begin{equation}\label{costante finale prova primo teorema}
k_\delta(\tau, \alpha, \beta,x-y) :=  \frac{\ii^{|\delta|} E_\delta(\tau,x-y) }{\alpha! \beta !}
\end{equation} we can set
\begin{equation}\label{definizione simbolo bM}
{\sigma_{\tau}}_M(v, \xi) := \sum_{|\alpha| + |\beta| < M} \sum_{|\alpha|+|\beta|\le|\delta | \le N\cdot(|\alpha| + |\beta|)}k_\delta(\tau, \alpha, \beta,x-y) \partial_x^\alpha \partial_y^\beta \partial_\xi^\delta a(v, v, \xi). 
\end{equation}
Note that for any $|\alpha|+|\beta|\le|\delta | \le N\cdot(|\alpha| + |\beta|)$, we have 
$$k_\delta(\tau, \alpha, \beta,x-y) \partial_x^\alpha \partial_y^\beta \partial_\xi^\delta a(v, v, \xi)  \in S^{m - |\delta|}_{\mu(|\alpha|+|\beta|)},$$
where $\mu$ is the order of $\tau$.

Indeed the element $k_{\delta}(\tau,\alpha,\beta,x-y)$ is composed, as shown in Lemma \ref{PRODPOLINOMI}, by at most $|\alpha|+|\beta|$ functions of $x-y$, each of which comes in our case  from the remainder of the Taylor expansion of $\tau_i$, $i^{th}$ component of $\tau$, which  has the form
\begin{equation}\label{REMINDERTAO}\frac{|\beta|}{\beta!}\int_{0}^{1}(1-t)^{|\beta|-1}\partial^{\beta}\tau_i(t(x-y))dt,\end{equation}
where $\beta$ is a multi-index of length $N.$  In view of Definition \ref{TAO AMMISSIBILE} this product can be majorised by $\langle x-y\rangle^{\mu(|\alpha|+|\beta|)}.$ Also when we differentiate $k_{\delta}(\tau,\alpha,\beta,x-y)$ with respect to some components of $x$ or $y$, let's say $y_j$,  we will get, due to the Leibniz rule, a sum of products of at most $|\alpha|+|\beta|$ functions of $x-y$. For example, the function corresponding to $\partial_{y_j}$ would have the form
$$\partial_{y_j}\frac{|\beta|}{\beta!}\int_{0}^{1}(1-t)^{|\beta|-1}\partial^{\beta}\tau_i(t(x-y))dt=\frac{|\beta|}{\beta!}\int_{0}^{1}(1-t)^{|\beta|-1}(-t)\partial^{\beta+\delta_j}\tau_i(t(x-y))dt,$$
that can be again dominated, when one passes to the modulus, by a constant multiple of $\langle x-y\rangle^{\mu}$. All the other functions will be identical to \eqref{REMINDERTAO}, therefore, the whole product will be again majorised by $\langle x-y\rangle^{\mu(|\alpha|+|\beta|)}$. Iterating this argument it follows that $k_{\delta}(\tau,\alpha,\beta,x-y)$ is always dominated by $\langle x-y\rangle^{\mu(|\alpha|+|\beta|)}$.

When we differentiate with respect to some $y_j$ the term $\partial_{x}^{\alpha}\partial_{y}^{\beta}\partial_{\xi}^{\delta}a(v,v,\xi)$ we get instead
\begin{align*}&\partial_{y_{j}}\partial_{x}^{\alpha}\partial_{y}^{\beta}\partial_{\xi}^{\delta}a(v,v,\xi)\\
&=\nabla(\partial_{x}^{\alpha}\partial_{y}^{\beta}\partial_{\xi}^{\delta}a(x,y,\xi))|_{x=y=v}&\cdot(\partial_{y_j}\tau_1(x-y),\dots,\partial_{y_j}\tau_i(x-y),\dots,\partial_{y_j}\tau_n(x-y),\\
&&,\partial_{y_j}\tau_1(x-y),\dots,\partial_{y_j}\tau_i(x-y),\dots,\partial_{y_j}\tau_n(x-y)),
\end{align*}
which is a sum of products of a derivative with respect to a spacial variable of the amplitude $\partial_x^{\alpha}\partial_{y}^{\beta}\partial_{\xi}^{\delta}a(x,y,\xi)$ evaluated in $x=y=v$ multiplied by some $\partial_{y_j}\tau_i(x-y)$. All of this can be dominated when one passes to the modulus by $\langle\xi\rangle^{m-|\delta|}\langle x-y\rangle^{\mu}$. A similar argument runs for $x_j$. Finally if one iterates this, it follows that 
$$|\partial_{\xi}^{\delta'}\partial_{x}^{\alpha'}\partial_{y}^{\beta'}(\partial_x^{\alpha}\partial_{y}^{\beta}\partial_{\xi}^{\delta}a(x,y,\xi))|\lesssim_{\alpha',\beta',\delta'}\langle\xi\rangle^{m-|\delta|-|\delta'|}\langle x-y\rangle^{\mu(|\alpha'|+|\beta'|)}.$$
\\
Since the derivative with respect to $\partial_{\xi}^{\delta'}\partial_{x}^{\alpha'}\partial_{y}^{\beta'}$ of $k_\delta(\tau, \alpha, \beta,x-y) \partial_x^\alpha \partial_y^\beta \partial_\xi^\delta a(v, v, \xi)$ will be a combination of the derivative of $k_\delta(\tau, \alpha, \beta,x-y)$ and those of $\partial_x^\alpha \partial_y^\beta \partial_\xi^\delta a(v, v, \xi)$, due to the Leibniz rule, 
we have that
\begin{align*}|\partial_{\xi}^{\delta'}\partial_{x}^{\alpha'}\partial_{y}^{\beta'}(k_\delta(\tau, \alpha, \beta,x-y)\partial_x^{\alpha}\partial_{y}^{\beta}\partial_{\xi}^{\delta}a(x,y,\xi))|\\
\lesssim_{\alpha',\beta',\delta'}\langle\xi\rangle^{m-|\delta|-|\delta'|}\langle x-y\rangle^{\mu(|\alpha'|+|\beta'|)}\langle x-y\rangle^{\mu(|\alpha|+|\beta|)}.\\ \lesssim_{\alpha',\beta',\delta'}\langle\xi\rangle^{m-|\delta|-|\delta'|}\langle x-y\rangle^{\mu(|\alpha|+|\beta|)(|\alpha'|+|\beta'|)}.\end{align*}
The last inequality is justified by the fact $|\alpha|+|\beta|+|\alpha'|+|\beta'|\le(|\alpha|+|\beta|)\cdot(|\alpha'|+|\beta'|)$.
\\
All of this together gives us that $k_\delta(\tau, \alpha, \beta,x-y) \partial_x^\alpha \partial_y^\beta \partial_\xi^\delta a(v, v, \xi)$ in $S^{m-|\delta|}_{\mu(|\alpha|+|\beta|)}$. Applying Lemma \ref{EQUIVALENZA} we have $k_\delta(\tau, \alpha, \beta,x-y) \partial_x^\alpha \partial_y^\beta \partial_\xi^\delta a(v, v, \xi)$, and thus $\sigma_{\tau_M}$, produce quantizations equivalent to that of symbols in $S^{m-|\delta|}_{0}$, up to any finite number of derivatives.

It remains to estimate the amplitude $r_M(x, y, \xi)$ defined in \eqref{definizione resto di Taylor teo Schubin}. 
We have, substituting again \eqref{polinomi alpha beta tau primo teorema} in \eqref{definizione resto di Taylor teo Schubin}, that
\begin{multline*}
r_{M}(x,y,\xi)=\frac{1}{(M-1)!}\sum_{|\alpha|+|\beta|=M}\sum_{|\alpha|+|\beta|\le|\delta | \le N\cdot(|\alpha| + |\beta|)} E_\delta(\tau,x-y) (x-y)^{\delta} \times \\
\times \int_0^1 (1 - t)^{M - 1} \partial_x^{\alpha} \partial_y^\beta a\Big(v - t \tau(- w) , v + t( - w - \tau(- w)), \xi \Big)\, d t.
\end{multline*}
Now the operator coming from  the term
$$E_{\delta}(\tau,x-y)(x-y)^{\delta}\int_0^1 (1 - t)^{M - 1} \partial_x^{\alpha} \partial_y^\beta a\Big(v - t \tau(- w) , v + t( - w - \tau(- w)), \xi \Big)\, d t\,$$
is equivalent to the operator coming from the term
\begin{multline*}
r_{M,\delta}(x,y,\xi) \\
= (-1)^{|\delta|}E_{\delta}(\tau,x-y)\int_0^1 (1 - t)^{M - 1} \partial_x^{\alpha} \partial_y^\beta \partial_{\xi}^{\delta}a\Big(v - t \tau(- w) , v + t( - w - \tau(- w)), \xi \Big)\, d t.
\end{multline*} 
Thus in order to study $r_{M}(x,y,\xi)$ it is enough to study each of $r_{M,\delta}(x,y,\xi)$.
For the latter, due to the fact we are using an amplitude from $A^m_d$, we have the following estimate  for the integrand
\begin{eqnarray*}
|\partial_{x}^{\alpha}\partial_{y}^{\beta}\partial_{\xi}^{\delta}a\Big(v - t \tau(- w) , v + t( - w - \tau(- w)), \xi \Big)|&\lesssim& \langle \xi \rangle^{m - |\delta|} \langle tw \rangle^{d\cdot( |\alpha| + |\beta| )}\\
&\lesssim&\langle \xi \rangle^{m - |\delta|} \langle w \rangle^{d\cdot( |\alpha| + |\beta| )}\\
&=&\langle \xi \rangle^{m - |\delta|} \langle x-y \rangle^{d\cdot( |\alpha| + |\beta| )},\end{eqnarray*}
which is thus uniform in $t\in[0,1]$.

In order to have a complete understanding of $r_{M,\delta}(x,y,\xi)$ we are left to analysing $E_{\delta}(\tau,x-y)$. Again in view of Lemma \ref{PRODPOLINOMI} we know that this is the product of at most $|\alpha|+|\beta|$ functions of $w$. These functions, as we have already observed, have the form
\begin{eqnarray*}
\int_{0}^{1}(1-t)^{|\beta|-1}\partial^{\beta}\tau_i(t(x-y)) dt,\qquad\mbox{ with }\beta \mbox{ a multindex of order }N.
\end{eqnarray*}
Since $\tau$ is admissible of order $\mu\geq 0$,
each of these integrals can be dominated in the following way
\begin{eqnarray*}
\int_{0}^{1}(1-t)^{|\beta|-1}\partial^{\beta}\tau_i(t(x-y)) dt\lesssim \langle t(x-y)\rangle^{\mu}\le \langle x-y\rangle^{\mu}.
\end{eqnarray*}

Therefore, we have  $|E_{\delta}(\tau,x-y)|\lesssim \langle x-y\rangle^{\mu(|\alpha|+|\beta|)}$ and, subsequently, also  $r_{M,\delta}(x,y,\xi)$ is in $A^{m-|\delta|}_{(\mu+d)(|\alpha|+|\beta|)}$, which gives that $r_{M}(x,y,\xi)$ is in 
$A^{m-|\delta|}_{(\mu+d)(|\alpha|+|\beta|)}\subset A^{m-(|\alpha|+|\beta|)}_{(\mu+d)(|\alpha|+|\beta|)}$.

Again using Lemma \ref{EQUIVALENZA} we have that the operator with the amplitude $r_{M}(x,y,\xi)$ is weakly equivalent to an operator with an amplitude in $A^{m-|\delta|}_{0}$, and in view of the $|\delta|$ going to infinity, to a smoothing one.
\end{proof}

As a corollary of Theorem \ref{teorema da amplitude a simboli}  and in view of the properties of pseudo-differential operators in H\"ormander's classes, we get the following corollary:

\begin{corollary}\label{PASSAGGIOAKNAKNW}
The operators with amplitudes in the class $A^m_d$ can be represented as the Weyl, Kohn-Nirenberg, or anti-Kohn-Nirenberg operators with symbols in $S^m$ modulo a smoothing pseudo-differential operator, with remainders with symbols in $A^{m-M}_{dM}$, for any $M\geq 1$.
\end{corollary}
We note that since the above quantizing functions are linear we use Theorem \ref{teorema da amplitude a simboli} with $\mu=0$.

As a further consequence we have a formula to pass from a $\tau_1$-representation of an operator $A$ to a $\tau_2$-representation of the same operator, for different admissible functions $\tau_1$ and $\tau_2$.

\begin{corollary}\label{COROLLARIO PASSAGGIO}
Let $\tau_1$ and $\tau_2$ be admissible quantizing functions of respective orders $\mu_1,\mu_2\geq 0$.
Given an operator with a $\tau_1$-symbol,
$$A_{\sigma_{\tau_1}}u(x)=\int_{\mathbb{R}^{n}}\int_{\mathbb{R}^{n}}
e^{\ii(x-y)\cdot\xi}\sigma_{\tau_1}(x+\tau_1(y-x),\xi)u(y)dyd\xi,$$
it is possible to represent it also as an operator with a $\tau_2$-symbol,  for any $N\in\N$, with its asymptotic expansion given by
\begin{multline}\label{EQ:taus}
\sigma_{\tau_2}(x+\tau_2(y-x),\xi) \\
\sim\sum_{\alpha,\beta\geq 0}\sum_{|\alpha| + |\beta|\le |\delta | \le N\cdot(|\alpha| + |\beta|)} k_\delta(\tau_2,\alpha,\beta,x-y)
\partial_{x}^{\alpha}\partial_{y}^{\beta}\partial_{\xi}^{\delta}
\left[\sigma_{\tau_1}(x+\tau_1(y-x),\xi)\right]|_{x+\tau_2(y-x)}.
\end{multline} 
Moreover, we have $k_0(\tau_2,0,0,x-y)=1$ (for $\alpha=\beta=\delta=0$).

In the above, if $\sigma_{\tau_1}\in S^m$, then also $\sigma_{\tau_1}\in S^m_{\mu_1,\tau_1}$, and \eqref{EQ:taus} is a weak asymptotic expansion with terms in in $S^{m-(|\alpha|+|\beta|)}_{\mu(|\alpha|+|\beta|)}$ and remainder in $A^{m-M}_{(\mu_1+\mu_2)M}$, for any $M\geq 1$.

\end{corollary}
The notation for the derivatives in \eqref{EQ:taus} means that we first differentiate the function $\sigma_{\tau_1}(x+\tau_1(y-x),\xi)$ by applying the operator $\partial_{x}^{\alpha}\partial_{y}^{\beta}\partial_{\xi}^{\delta}$ to it, and then we plug in $x+\tau_2(y-x)$ in the place of the first variable of $\sigma_{\tau_1}(\cdot,\xi).$
\proof
We can set as an amplitude $a(x,y,\xi)$ the symbol $\sigma_{\tau_1}(x+\tau_1(y-x),\xi)$, and then apply Theorem \ref{teorema da amplitude a simboli}.
\endproof

We can also compute the adjoint and the transpose of an operator $A_{\sigma,{\tau}}$ and give an asymptotic expansion of these two new operators in terms of $\sigma$ and $\tau$. We will adopt the notation of functions $k_\delta(\tau,\alpha,\beta,x-y)$ arising from Theorem \ref{teorema da amplitude a simboli}.

\begin{proposition}\label{ADJOINT}
Let $\tau$ be an admissible quantizing function of order $\mu\geq 0$. Let $\sigma\in S^m$ and let $A_{\sigma,{\tau}}$ be the corresponding $\tau$-quantized operator as in \eqref{EQ:quant}. Let us denote
\begin{equation}\label{EQ:newtau}
 \tau^{*}(z):=z+\tau(-z),\quad z\in\Rn,
\end{equation}
which is also an admissible quantizing function of the same order $\mu$.
Then the transpose operator $A^{t}_{\sigma,{\tau}}$ is the operator with the amplitude given by $\sigma(x+\tau^{*}(y-x)),-\xi)$ with the quantizing function $\tau^*$, i.e. 
$$
A_{\sigma,{\tau}}^{t}v(x)=\int_{\mathbb{R}^{n}}\int_{\mathbb{R}^{n}}e^{\ii(x-y)\cdot\xi}\sigma(x+\tau^{*}(y-x),-\xi) v (y) dyd\xi.
$$
We can view it as a $\tau^*$-quantized operator with the symbol $\sigma_t(x,\xi)=\sigma(x,-\xi)$, i.e.
$$
A_{\sigma,{\tau}}^{t}=A_{\sigma_t,\tau^*}.
$$
Furthermore, $A_{\sigma,{\tau}}^{t}$ can be also seen as a $\tau$-quantization with a symbol $\sigma'$: denoting $\sigma_{\tau^t}(x,y,\xi):=\sigma(x+\tau^{*}(y-x),-\xi)$, for any $N\in\N$ it has the weak asymptotic expansion
\begin{multline}\label{EQ:tauadj}
\sigma'(x+\tau(y-x),\xi) \\
\sim\sum_{\alpha,\beta\geq 0}\sum_{|\alpha| + |\beta|\le |\delta |\le N\cdot(|\alpha| + |\beta|)} k_\delta(\tau,\alpha,\beta,x-y)\partial_{x}^{\alpha}\partial_{y}^{\beta}\partial_{\xi}^{\delta}\sigma_{\tau^{t}}(v,v,-\xi)|_{x+\tau(y-x)},
\end{multline} 
with terms in $S^{m-(|\alpha|+|\beta|)}_{\mu(|\alpha|+|\beta|)}$ and remainder in $A^{m-M}_{2\mu M}$, for any $M\geq 1$.

The adjoint $A_{\sigma,{\tau}}^{*}$ of $A_{\sigma,{\tau}}$ is the operator with the symbol $\overline{\sigma}$ but with the quantization function $\tau^*$, i.e.
$$
A_{\sigma,{\tau}}^{*}=A_{\overline{\sigma},\tau^*}.
$$
Furthermore, $A_{\sigma,{\tau}}^{*}$ can be also seen as a $\tau$-quantization with the symbol $\sigma''$: denoting $\sigma_{\tau^*}(x,y,\xi):=\overline{\sigma(x+\tau^{*}(y-x),\xi)}$, for any $N\in\N$ it has the asymptotic expansion
\begin{multline}\label{EQ:tauadj2}
\sigma''(x+\tau(y-x))\\
\sim\sum_{\alpha,\beta\geq 0}\sum_{|\alpha|+|\beta|\le |\delta |\le N\cdot(|\alpha| + |\beta|)}k_\delta(\tau,\alpha,\beta,x-y)\partial_{x}^{\alpha}\partial_{y}^{\beta}\partial_{\xi}^{\delta}{\sigma_{\tau^{*}}(v,v,\xi)}|_{x+\tau(y-x)},
\end{multline} 
with terms in $S^{m-|\alpha|-|\beta|}_{\mu(|\alpha|+|\beta|)}$ and remainder in $A^{m-M}_{2\mu M}$, for any $M\geq 1$.
\end{proposition}
\proof
We want to find an operator $A_{\sigma,{\tau}}^{t}$ such that
$$\langle A_{\sigma,{\tau}}u,v\rangle=\langle u,A_{\sigma,{\tau}}^{t}v\rangle \qquad u,v\in\mathcal{S}(\mathbb{R}^{d}),$$
in the sense of the distributional pairing. Using 
a change of variable $\xi\mapsto (-\xi)$ we can write
\begin{align*}\langle A_{\sigma,{\tau}}u,v\rangle=\int_{\mathbb{R}^{n}}\int_{\mathbb{R}^{n}}\int_{\mathbb{R}^{n}}e^{\ii(x-y)\cdot\xi}\sigma(x+\tau(y-x),\xi)u(y)v(x)dxdyd\xi&=&\\
\int_{\mathbb{R}^{n}}u(y)dy\int_{\mathbb{R}^{n}}\int_{\mathbb{R}^{n}}e^{\ii(x-y)\cdot\xi}\sigma(y+x-y+\tau(y-x),\xi)v(x)dxd\xi&=&\\
\int_{\mathbb{R}^{n}}u(y)dy\int_{\mathbb{R}^{n}}\int_{\mathbb{R}^{n}}e^{\ii(y-x)\cdot\xi}\sigma(y+\tau^{*}(x-y),-\xi)v(x)dxd\xi=\langle u,A_{\sigma,{\tau}}^{t}v\rangle.
\end{align*}
From this we see that the transpose operator $A_{\sigma,{\tau}}^{t}$ of $A_{\sigma,\tau}$ is the operator with symbol 
$\sigma(y+\tau^{*}(x-y),-\xi)$ with $\tau^{*}$ given by \eqref{EQ:newtau}. 
Consequently $A^{t}_{\sigma,{\tau}}$ can be rewritten in the following form switching $x$ and $y$:
$$
A_{\sigma,{\tau}}^{t}v(x)=\int_{\mathbb{R}^{n}}\int_{\mathbb{R}^{n}}e^{\ii(x-y)\cdot\xi}\sigma(x+\tau^{*}(y-x),-\xi) v (y) dyd\xi,
$$
and that $\tau^{*}$ still verifies the condition of Definition \ref{TAO AMMISSIBILE}. By 
Theorem \ref{teorema da amplitude a simboli} viewing it as an amplitude operator we can also view it as a $\tau^*$-quantization with the symbol having the asymptotic expansion given by \eqref{EQ:tauadj}.

Similarly, for the adjoint operator $A^*_{\sigma,\tau}$ we use the equality
\begin{align*}
\int_{\mathbb{R}^{d}}\int_{\mathbb{R}^{d}}\int_{\mathbb{R}^{d}}e^{\ii(x-y)\cdot\xi}\sigma(x+\tau(y-x),\xi)u(y)\overline{v(x)}dxdyd\xi&=&
\\ \int_{\mathbb{R}^{d}}u(y)dy\overline{\int_{\mathbb{R}^{d}}\int_{\mathbb{R}^{d}}e^{\ii(y-x)\cdot\xi}\overline{\sigma(x+\tau(y-x),\xi)}v(x)dxd\xi}&=&\\
\int_{\mathbb{R}^{d}}u(y)dy\overline{\int_{\mathbb{R}^{d}}\int_{\mathbb{R}^{d}}e^{\ii(y-x)\cdot\xi}\overline{\sigma(y+x-y+\tau(y-x),\xi)}v(x)dxd\xi}&=&\\
\int_{\mathbb{R}^{d}}u(y)dy\overline{\int_{\mathbb{R}^{d}}\int_{\mathbb{R}^{d}}e^{\ii(y-x)\cdot\xi}\overline{\sigma(y+\tau^{*}(x-y),\xi)}v(x)dxd\xi}.
\end{align*}
We see from this that the adjoint 
$A_{\sigma,\tau}^*$ is a $\tau^*$-quantized pseudo-differential operator (i.e. its quantization function is $\tau^*$ and symbol $\overline{\sigma}$.  
By 
Theorem \ref{teorema da amplitude a simboli} viewing it as an amplitude operator we can also view it as a $\tau^*$-quantization with the symbol having the asymptotic expansion given by \eqref{EQ:tauadj2}.
\endproof

We are now ready to prove also the composition formula for $\tau$-quantized symbols. For this, we will freely rely on Corollary \ref{COROLLARIO PASSAGGIO} allowing one to change quantizations whenever it is convenient.

\begin{theorem}\label{COMPOSIZIONE}
Let $\tau_1,\tau_2,\tau_3$ be admissible quantizing function of corresponding orders $\mu_1,\mu_2,\mu_3\geq 0$. Let
$A_1$, $A_2$ be two operators associated respectively to two $\tau_1$- and $\tau_2$-quantized symbols in $S^{m_1}$ and $S^{m_2}$. Then their composition $A_1\circ A_2$ can be viewed as a $\tau_3$-quantized operator with the symbol $\sigma$, for any $N\in\N$
having the weak asymptotic expansion
\begin{multline*}
\sigma(x+\tau_3(y-x),\xi)\sim\\ \sum_{\alpha,\beta\geq 0}\sum_{|\alpha|+|\beta|\leq |\delta|\leq N(|\alpha|+|\beta|)}\sum_{\gamma+\epsilon=\delta}k_{\delta}(\tau,\alpha,\beta,\gamma,\epsilon,x-y)[\partial_{\xi}^{\gamma}\partial_{x}^{\alpha}\sigma'_{KN}(x,\xi)\partial_{\xi}^{\epsilon}\partial_{y}^{\beta}\sigma''_{AKN}(y,\xi)]|_{x=y=v}.
\end{multline*}
with respective terms in $S^{m_1+m_2-(|\alpha|+|\beta|)}_{\mu_3(|\alpha|+|\beta|)}$ and remainder in $A^{m_1+m_2-M}_{(\mu_1+\mu_2+\mu_3)M}$, for any $M\geq 1$.
In the above formula, at the end we plug in $v=x+\tau_3(y-x)$ and, $\sigma_{KN}^{'}$ is the symbol one gets when expressing $A_1$ as an operator with Kohn-Nirenberg symbol as in Corollary \ref{COROLLARIO PASSAGGIO} and, similarly, $\sigma_{AKN}''$ is the symbol one gets when expressing $A_2$ as an operator with the anti-Kohn-Nirenberg symbol. Moreover, $k_\delta^{'}(\tau_3, \alpha, \beta,\gamma,\epsilon,x-y)= \frac{\delta!}{\gamma!\epsilon!} k_\delta(\tau_3, \alpha, \beta,x-y)$, with $k_\delta(\tau_3, \alpha, \beta,x-y)$ as in Theorem \ref{teorema da amplitude a simboli}.

In particular, if we take $\tau_1=\tau_2=\tau_3$ we see that the operator classes ${\rm OP}_\tau^0$ and $\cup_{m\in\R} {\rm OP}_\tau^m$ form algebras of operators, modulo remainders.
\end{theorem}
\proof
The proof follows almost verbatim \cite{Schubin} in the first part, while the second part has changes due to our different quantization. 
By using Corollary \ref{PASSAGGIOAKNAKNW} and Corollary \ref{COROLLARIO PASSAGGIO}, let us represent our operators $A_1$ and $A_2$ via Kohn-Nirenberg and anti-Kohn-Nirenberg symbols, respectively, modulo some errors which are smoothing operators, so that we can write
$$A_1v(x)=\int_{\mathbb{R}^{n}}\int_{\mathbb{R}^{n}} e^{\ii(x-y)\cdot\xi}\sigma_{KN}^{'}(x,\xi)v(y)dyd\xi+E'v(x)=:Iv(x)+E'v(x),$$
and
$$A_2u(x)=\int_{\mathbb{R}^{n}}\int_{\mathbb{R}^{n}} e^{\ii(x-y)\cdot\xi}\sigma_{AKN}^{''}(y,\xi)u(y)dyd\xi +E''u(x)=:IIu(x)+E''u(x),$$
where $E', E''$ are smoothing remainders as in Corollary \ref{PASSAGGIOAKNAKNW}.
Consequently, $Iv(x)$ can be rewritten as
\begin{equation}
Iv(x)=\int_{\mathbb{R}^{n}}e^{\ii x\cdot\xi}\sigma_{KN}^{'}(x,\xi)\hat{v}(\xi)d\xi,
\end{equation}
while the Fourier transform of $IIu(x)$ is 
$\widehat{IIu} (\xi)=\int_{\mathbb{R}^{n}}e^{-\ii y\cdot\xi}\sigma^{''}_{AKN}(y,\xi)u(y)dy$.
If we compose $I$ and $II$ we then have
\begin{eqnarray*}
I(IIu(x))&=&\int_{\mathbb{R}^{n}}e^{\ii x\cdot\xi}\sigma_{KN}^{'}(x,\xi)\widehat{IIu}(\xi)d\xi\\
&=&\int_{\mathbb{R}^{n}} \int_{\mathbb{R}^{n}}e^{\ii(x-y)\cdot\xi}\sigma_{KN}^{'}(x,\xi)\sigma_{AKN}^{''}(y,\xi)u(y)dyd\xi.
\end{eqnarray*}
Therefore we have that $I\circ II$ is the operator with the amplitude $\sigma_{KN}^{'}(x,\xi)\sigma_{AKN}^{''}(y,\xi)$ of order $m_1+m_2$. In turn, this can be also written as a $\tau_3$-quantized symbol, $\sigma(x+\tau_3(y-x))$, with the asymptotic expansion given by Theorem \ref{teorema da amplitude a simboli}. More precisely,
with $v=x+\tau_3(y-x)$, we have
\begin{eqnarray*}
&&  \sigma(x+\tau_3(y-x), \xi) \\
&\sim& \sum_{\alpha,\beta\geq 0}\sum_{|\alpha| + |\beta|\le|\delta | \le N\cdot(|\alpha| + |\beta|)}k_\delta(\tau_3,\alpha,\beta,x-y)[\partial_{x}^{\alpha}\partial_{y}^{\beta}\partial_{\xi}^{\delta}\left(\sigma_{KN}^{'}(x,\xi)\sigma_{AKN}^{''}(y,\xi)\right)]_{x=y=v}\\
&=&\sum_{\alpha,\beta\geq 0}\sum_{|\alpha| + |\beta|\le|\delta | \le N\cdot(|\alpha| + |\beta|)}k_\delta(\tau_3, \alpha, \beta,x-y)  \partial_\xi^\delta [\partial_x^\alpha \sigma_{KN}^{'}(x,\xi) \partial_y^\beta \sigma_{AKN}^{''}(y,\xi)]|_{x=y=v}\\
&=&\sum_{\alpha,\beta\geq 0}\sum_{|\alpha| + |\beta|\le|\delta | \le N\cdot(|\alpha| + |\beta|)}k_\delta(\tau_3, \alpha, \beta,x-y)\times 
\\& & \times \sum_{\gamma+\epsilon=\delta}\frac{\delta!}{\gamma!\epsilon!} [\partial_\xi^\gamma \partial_x^\alpha \sigma_{KN}^{'}(x,\xi)\partial_\xi^\epsilon \partial_y^\beta \sigma_{AKN}^{''}(y,\xi)]\\
&=&\sum_{\alpha,\beta\geq 0}\sum_{|\alpha| + |\beta|\le|\delta | \le N\cdot(|\alpha| + |\beta|)}
\sum_{\gamma+\epsilon=\delta}
k_\delta^{'}(\tau_3, \alpha, \beta,\gamma,\epsilon,x-y)\times\\ & & \times  [\partial_\xi^\gamma \partial_x^\alpha \sigma_{KN}^{'}(x,\xi)\partial_\xi^\epsilon \partial_y^\beta \sigma_{AKN}^{''}(y,\xi)]_{x=y=v},
\end{eqnarray*}
where we have used  in the third line Leibniz formula with respect to $\xi$ and also $k_\delta^{'}(\tau_3, \alpha, \beta,\gamma,\epsilon,x-y):=k_\delta(\tau_3, \alpha, \beta,x-y)\cdot \frac{\delta!}{\gamma!\epsilon!}$ in the last equality.
\endproof

We now discuss elliptic $\tau$-quantized operators and their parametrices.

\begin{proposition}\label{PROP:parametrix}
Let $\tau$ be an admissible quantizing function of order $\mu\geq 0$, and let $\sigma\in S^{m}$ be an elliptic symbol, that is, 
$$|\sigma(x,\xi)|\ge C(1+|\xi|)^{m} \; \textrm{ for all } \; x,\xi\in\Rn, \; |\xi|\ge R_0,$$
for some $R_0>0$.
Then there exists a symbol $\kappa\in S^{-m}$ such that 
\begin{equation}\label{parametrix}
A_{\kappa,{\tau}}A_{\sigma,{\tau}}=I+R
\end{equation}
and
\begin{equation}\label{parametrix2}
A_{\sigma,{\tau}}A_{\kappa,{\tau}}=I+S,
\end{equation}
where $A_{\sigma, {\tau}}$ and $A_{\kappa,{\tau}}$
are $\tau$-quantized operators as in \eqref{EQ:quant},
and $R$ and $S$ are smoothing pseudo-differential operators. 
\end{proposition}

\begin{proof}
We will prove \eqref{parametrix} since the proof of \eqref{parametrix2} is analogous.
In view of Theorem \ref{teorema da amplitude a simboli} we can write
\begin{equation}\label{PASSAGGIO}
A_{\sigma,{\tau}}=A_{\sigma_1,{KN}}+R_1,
\end{equation}
where $R_1$ is a smoothing operator and $A_{\sigma_1,{KN}}$ is the Kohn-Nirenberg quantized pseudo-differential operator with the corresponding symbol $\sigma_1\in S^m$ with the asymptotic expansion given in Corollary \ref{COROLLARIO PASSAGGIO}. It follows also that the symbol $\sigma_1$ is elliptic, so that using the usual parametrix for Kohn-Nirenberg quantization (see e.g.  \cite{Lerner}, \cite{Ruz}), there exists an operator $A_{\kappa_1,{KN}}$ such that
\begin{equation}\label{PARAMETRICA}
A_{\kappa_1,{KN}}A_{\sigma_1,{KN}}=I+R_2,
\end{equation}
for some $\kappa_1\in S^{-m}$, and where $R_2$ is a smoothing operator.
Again applying Corollary \ref{COROLLARIO PASSAGGIO}  to $A_{\kappa_1,KN}$ we have
$$A_{\kappa,{\tau}}=A_{\kappa_1,{KN}}+R_3,$$
for some $\kappa\in S^{-m}$.
Using this together with \eqref{PASSAGGIO} and \eqref{PARAMETRICA}, we have
\begin{eqnarray*}
A_{\kappa,{\tau}}A_{\sigma,{\tau}}&=&(A_{\kappa_1,{KN}}+R_3)(A_{\sigma_1,{KN}}+R_1)\\
&=&A_{\kappa_1,{KN}} A_{\sigma_1,{KN}}+A_{\kappa_1,{KN}}R_1+R_3 A_{\sigma_1,{KN}} +R_3 R_1\\
&=&I+R_2+A_{\kappa_1,{KN}}R_1+R_3 A_{\sigma_1,{KN}} +R_3 R_1\\
&=&I+R,
\end{eqnarray*}
where we have set $R:=R_2+A_{\kappa_1,{KN}}R_1+R_3 A_{\sigma_1,{KN}} +R_3 R_1$, which has the symbol in $S^{-\infty}$.
\end{proof}

Similar to the proof of Proposition \ref{PROP:parametrix}, using Corollary \ref{COROLLARIO PASSAGGIO} we can extend other results for pseudo-differential operators to the $\tau$-quantizations and to $S^m_{\mu,\tau}$-classes. Let us give the G\r{a}rding inequality as an example.
 
\begin{proposition}[G\r{a}rding inequality]
Let $\tau$ be an admissible quantizing function as in Definition \ref{TAO AMMISSIBILE}, and assume that  a symbol  $\sigma\in S^{2m}$, $m\in\mathbb R$, satisfies
\begin{equation}\label{EQ:Garding-condition}
{\rm Re}\, \sigma(x,\xi)\ge C(1+|\xi|)^{2m},\qquad |\xi|\ge R,
\end{equation} 
for some $R>0$ and all $x\in\Rn$.
Then for any $s\in\R$ there exist positive constants $C_1$ and $C_2$ such that 
\begin{equation}\label{EQ:Garding-est}
 {\rm Re}\, (A_{\sigma,{\tau}}u,u)\ge C_1\|u\|^{2}_{H^m}-C_2\|u\|^{2}_{H^s},
\end{equation}
for all $u\in H^m(\Rn)$. 
\end{proposition}
\begin{proof}
First, using Corollary \ref{COROLLARIO PASSAGGIO} we can rewrite $A_{\sigma,{\tau}}$ as a Kohn-Nirenberg operator with some symbol $\sigma_1\in S^{2m}$. We note that the ellipticity condition is preserved by such transformation. Consequently, the estimate \eqref{EQ:Garding-est} follows from the usual G\r{a}rding inequality (see e.g. \cite[Theorem 6.1]{Taylor}), with the remainder only influencing the constant $C_2$.
\end{proof}

\section{$\tau$-quantizations with bounded derivatives}\label{41}

The special case of admissible quantizing functions from Definition \ref{TAO AMMISSIBILE} are the functions that have growth of the first order in $x$, while still being (possibly) nonlinear. In this case one can always work  with the usual H\"ormander classes $S^m$ without going to $S^m_d$. For such quantized functions we set $\mu=0$ in Definition \ref{TAO AMMISSIBILE}.
For clarity, let us make this explicit:

\begin{definition}\label{TAO AMMISSIBILE2}
A function $\tau:\mathbb{R}^{n}\rightarrow\mathbb{R}^{n}$ will be called {\em admissible with bounded derivatives} if $\tau\in C^{\infty}$ and $\tau(0)=0$, and if 
\begin{equation}
\sup_{x\in\Rn}|\partial_{x}^{\alpha}\tau(x)| <\infty
\end{equation}
holds for all multi-indices $\alpha\in\mathbb{N}^{n}_0\backslash\{0\}$. 
\end{definition}
We note that by the Taylor expansion formula these assumptions actually imply that $|\tau(x)|\lesssim\langle x\rangle.$

Now if we compose our $\tau$ function as in Definition \ref{TAO AMMISSIBILE2} with a symbol in $S^m$ we get an amplitude in $A^m=A^m_0$, i.e. if $\tau$ is admissible with bounded derivatives and $\sigma \in S^m$, then the function $a(x, y, \xi) := \sigma (x + \tau(y- x), \xi)$ belongs to the class $A^m$. This follows by the repeated application of the chain rule.

In this section we are going to discuss another approach to move from a $\tau$-quan\-ti\-zation  of an operator with symbol in $S^{m}$ to another quantization. More precisely, instead of passing from one quantization to another through an amplitude operator as we did in the proof of Corollary \ref{COROLLARIO PASSAGGIO}, we are going to consider the operator
\begin{equation*}\label{pseudo}
Au(x)=\int_{\R^n}\int_{\R^n}e^{\ii(x-y)\cdot\xi}\sigma(x+\tau(y-x),\xi)u(y)dyd\xi
\end{equation*}
and to perform a change of variable with respect to the space variable in the function 
$\sigma(x+\tau(y-x),\xi)$
inside the integral.
To do so it is natural to consider $x$ fixed and to ask for the invertibility of the mapping 
\begin{eqnarray}\label{taustar}
\tau_x:y\mapsto x+\tau(y-x).
\end{eqnarray}
To have $\tau_x$ invertible we can use the following criterion, see e.g.  \cite{RuzSug} and historical references therein, for this remarkable property:

\begin{theorem}[Hadamard]\label{THM:Hadamard}
A $C^1$ map $f:\mathbb{R}^{n}\rightarrow\mathbb{R}^{n}$ is a global $C^1$-diffeomorphism if and only if the Jacobian $\det \frac{\partial f(x)}{\partial x}$ never vanishes and $|f(x)|\rightarrow\infty$ whenever $|x|\rightarrow\infty$.
\end{theorem}

In our case, in order to have $\tau_x$ invertible we then ask for $|\tau_x(y)|\rightarrow\infty$ whenever $|y|\rightarrow\infty$, which is equivalent to asking that $|\tau(z)|\rightarrow\infty$ whenever $|z|\rightarrow\infty$.
One can also readily check the other properties in the following

\begin{corollary}\label{COR:hadamard-tau}
Let $\tau:\mathbb{R}^{n}\rightarrow\mathbb{R}^{n}$ be an admissible quantizing function with bounded derivatives. Assume that the Jacobian $\det \frac{\partial \tau(x)}{\partial x}$ never vanishes and that $|\tau(x)|\rightarrow\infty$ whenever $|x|\rightarrow\infty$. 
Then both $\tau$ and $\tau_x$ for any $x\in\Rn$ are global diffeomorphisms, with their inverses satisfying $|\tau^{-1}(y)|\to\infty$ and $|\tau_x^{-1}(y)|\to\infty$ whenever $|y|\to\infty.$ 

Moreover, if $\inf_{x\in\Rn} |\det \frac{\partial \tau(x)}{\partial x}|\not=0$ then $\tau^{-1}$ is also an admissible quantizing function with bounded derivatives.
\end{corollary} 

Using the above observations we have the following property:

\begin{proposition}\label{PROPCAMBIO}
Let $\tau$ be an admissible quantizing function with bounded derivatives such that $\inf_{x\in\Rn} |\det \frac{\partial \tau(x)}{\partial x}|\not=0$ and $|\tau(x)|\rightarrow\infty$ whenever $|x|\rightarrow\infty$. 
Let $\sigma\in S^m$.
Then the operator
\begin{equation}\label{EQ:tauop1}
A_{\sigma,\tau}u(x)=\int_{\R^n}\int_{\R^n}e^{\ii(x-y)\cdot\xi}\sigma(x+\tau(y-x),\xi)u(y)dyd\xi
\end{equation} 
can be written in the Kohn-Nirenberg type form composed with the change of variables, i.e.
\begin{equation}\label{EQ:tauop2}
A_{\sigma,\tau}u(x)=\int_{\R^n}\int_{\R^n}e^{\ii(x-y)\cdot\xi}b(x,\xi)u(\tau_x^{-1})(y)dyd\xi+Eu(x),
\end{equation} 
for some $b\in S^m$, where $E$ is a smoothing operator.
\end{proposition}

\proof
Setting $w:=x+\tau(y-x)$ we have
\begin{eqnarray*}
w-x&=&\tau(y-x),\\
-\tau^{-1}(w-x)&=&x-y.\label{beta}
\end{eqnarray*}
Using the function $\tau_x(y)=x+\tau(y-x)$ from \eqref{taustar}, we also have
$$
\tau_x^{-1}(w)=y.
$$
Since $\tau_x(x)=x+\tau(0)=x$, we also have $\tau_x^{-1}(x)=x$ in view of 
Theorem \ref{THM:Hadamard}.
Inserting all of these in \eqref{EQ:tauop1}
we get
\begin{equation*}
A_{\sigma,\tau}u(x)=\int_{\R^n}\int_{\R^n}
e^{\ii(\tau_x^{-1}(x)-\tau_x^{-1}(w))\cdot\xi}\sigma(w,\xi) u(\tau_x^{-1}(w))\left|\det\left(\frac{\partial\tau_x^{-1}(w)}{\partial w}\right)\right|dwd\xi.
\end{equation*}
Now we can observe that the main contribution here is from the $w$ close to $x$. Indeed, let $\chi(x,w)$ be a cut off function supported in a small neighborhood of $w=x$ and such that $\chi(x,x)=1.$ Then we can write 
\begin{multline*}
A_{\sigma,\tau}u(x) \\
=\underbrace{\int_{\R^n}\int_{\R^n}
e^{\ii(\tau_x^{-1}(x)-\tau_x^{-1}(w))\cdot\xi}\sigma(w,\xi) 
(1-\chi(x,w))u(\tau_x^{-1}(w)) \left|\det\left(\frac{\partial\tau_x^{-1}(w)}{\partial w}\right)\right|
dwd\xi}_{I_1}
\\
+\underbrace{\int_{\R^n}\int_{\R^n}e^{\ii(\tau_x^{-1}(x)-\tau_x^{-1}(w))\cdot\xi}\sigma(w,\xi) \chi(x,w)
u(\tau_x^{-1}(w))
\left|\det\left(\frac{\partial\tau_x^{-1}(w)}{\partial w}\right)\right|
dwd\xi}_{I_2}.
\end{multline*} 
For $w$ away from $x$ we can integrate by parts in $\xi$ using the property
$$
\partial_\xi^{\alpha}e^{\ii(\tau_x^{-1}(x)-\tau_x^{-1}(w))\cdot\xi}=\ii^{|\alpha|}
(\tau_x^{-1}(x)-\tau_x^{-1}(w))^{\alpha}
e^{\ii(\tau_x^{-1}(x)-\tau_x^{-1}(w))\cdot\xi},
$$
so that we have
\begin{multline*}
I_1=
(-\ii)^{|\alpha|}\int_{\R^n}\int_{\R^n}e^{\ii(\tau_x^{-1}(x)-\tau_x^{-1}(w))\cdot\xi}(1-\chi(x,w))\left[(\tau_x^{-1}(x)-\tau_x^{-1}(w))^{\alpha}\right]^{-1} \times \\
\times \partial_\xi^{\alpha}\sigma(w,\xi)
u(\tau_x^{-1}(w))
\left|\det\left(\frac{\partial\tau_x^{-1}(w)}{\partial w}\right)\right| dwd\xi.
\end{multline*} 
Since $\inf_{x\in\Rn} |\det \frac{\partial \tau(x)}{\partial x}|\not=0$, for any $w$ in the support of the integrand of $I_1$, we can find $\alpha$ with $|\alpha|$ as large as we want such that $|(\tau_x^{-1}(x)-\tau_x^{-1}(w))^{\alpha}|$ is uniformly bounded away from $0$. 
By Corollary \ref{COR:hadamard-tau} and since $\sigma\in S^m$, we have that 
$\left[(\tau_x^{-1}(x)-\tau_x^{-1}(w))^{\alpha}\right]^{-1} \partial_\xi^{\alpha}\sigma(w,\xi)$ is in $S^{{m}-|\alpha|}$, so that the term $I_1$ is smoothing. 

For $I_2$, since $w$ and $x$ are close to each other, we can write 
$$\tau_x^{-1}(x)-\tau_x^{-1}(w)=L_{xw}(x-w),$$
where $L_{xw}$ is a linear mapping which depends smoothly on $x,w$.
Now we can write
\begin{multline*}
I_2=\int_{\R^{n}}\int_{\R^{n}}e^{\ii L_{xw}(x-w)\cdot\xi}\chi(x,w)\sigma(w,\xi)u(\tau_x^{-1}(w))
\left|\det\left(\frac{\partial\tau_x^{-1}(w)}{\partial w}\right)\right| dwd\xi
\\ =\int_{\R^{n}}\int_{\R^{n}}e^{\ii (x-w)\cdot\xi}\chi(x,w)\sigma(w,L'_{xw}\xi)
u(\tau_x^{-1}(w))
\left|\det\left(\frac{\partial\tau_x^{-1}(w)}{\partial w}\right)\right|
 \left|\det L_{xw}^{-1}\right| dwd\xi,
\end{multline*} 
where $L_{xw}'={\left[{L_{xw}}^t\right]}^{-1}$.
From now on, by modifying a well-known reduction for amplitude operators, see e.g. 
\cite[Theorem 2.5.8]{Ruz}, we see that the pseudo-differential operator with the amplitude 
$$a(x,w,\xi)=\chi(x,w)\sigma(w,L'_{xw}\xi)
\left|\det\left(\frac{\partial\tau_x^{-1}(w)}{\partial w}\right)\right|
 \left|\det L_{xw}^{-1}\right| $$ 
can be reduced to a Kohn-Nirenberg operator with symbol $b(x,\xi)$ with the asymptotic expansion of the form
$$b(x,\xi)\sim \sum_{\alpha\geq 0}\frac{\ii^{-|\alpha|}}{\alpha!}\partial_{\xi}^{\alpha}\partial_{w}^{\alpha}a(x,w,\xi)|_{w=x},$$
completing the proof.
\endproof

\section{Calder\'on-Vaillancourt theorem}
\label{SEC:5}

In this section we discuss the $L^2$-boundedness of the appearing operators. It is convenient to look at this problem from the point of view of more general amplitude operators.
First we note that Lemma \ref{EQUIVALENZA} and the $L^2$-boundedness of the pseudo-differential operators with symbols in the H\"ormander class $S^0$ immediately imply the following property:

\begin{corollary}
Let $a\in A^{0}_d$ for some $d\in\R$. Then the associated amplitude operator
$$Au(x)=\int_{\mathbb{R}^{n}}\int_{\mathbb{R}^{n}}e^{\ii(x-y)\xi}a(x,y,\xi)u(y)dyd\xi$$
extends to a bounded operator from $L^2(\Rn)$ to $L^2(\Rn)$. 

Consequently, if $\tau$ is an admissible quantizing function (of any order $\mu\geq 0$) and $\sigma\in S^m$, then the operator $A_{\sigma,\tau}$ defined by
$$
A_{\sigma,{\tau}} u(x) =\int_{\R^n} e^{\ii x \cdot \xi} \sigma(x+\tau(y-x), \xi) \widehat u(\xi) \,d \xi
$$
extends to a bounded operator from $L^2(\Rn)$ to $L^2(\Rn)$. 
\end{corollary}

However, an interesting and important question is how many derivatives of the symbol/amplitude should be bounded to ensure the $L^2$-boundedness of pseudo-differential operators in the spirit of the celebrated Calder\'on-Vaillancourt theorem \cite{CV}.
Similar results for classes of Fourier integral operators have been obtained in 
 \cite{RuzSug1}. We now prove a variant of the Calder\'on-Vaillancourt theorem for amplitude operators.

\begin{theorem}\label{THM:CV}
Let $a=a(x,y,\xi):\Rn\times\Rn\times\Rn\to\mathbb C$ be such that  
\begin{equation}\label{EQ:CV-assump}
\sup_{x,y,\xi\in\Rn} |\partial_x^{\alpha}\partial_{y}^{\beta}\partial_\xi^{\gamma}a(x,y,\xi)|<\infty
\end{equation} 
holds for all multi-indices $\alpha,\beta,\gamma$ such that $|\alpha|,|\beta|,|\gamma|\le 2n+1$. Then the operator 
$$Au(x)=\int_{\mathbb{R}^{n}}\int_{\mathbb{R}^{n}}e^{\ii(x-y)\xi}a(x,y,\xi)u(y)dyd\xi$$
extends to a bounded operator from $L^2(\Rn)$ to $L^2(\Rn)$. Moreover, we have
\begin{equation}\label{EQ:Cvampest}
\|A\|_{L^2\rightarrow L^2}\le C\sup_{|\alpha|,|\beta|,|\gamma|\leq 2n+1}\sup_{x,y,\xi\in\Rn} |\partial_x^{\alpha}\partial_{y}^{\beta}\partial_\xi^{\gamma}a(x,y,\xi)|.
\end{equation} 
\end{theorem}
\proof
We follow the strategy of the proof of \cite[Theorem 2.1]{RuzSug1}.
Let $\chi\in C_0^\infty(\Rn)$ be a real-valued non-negative functions such that the family
of shifts, $\chi_k(x)=\chi(x-k)$, $k\in\mathbb Z^n$, forms a partition of unity. Then we decompose
$$
A=\sum_{i,j,k\in\mathbb Z^n} A_{i,j,k},
$$
with the operators
$$A_{i,j,k}u(x)=\int_{\mathbb{R}^{n}}\int_{\mathbb{R}^{n}} e^{\ii(x-y)\cdot\xi}a(x,y,\xi)\chi_i{(x)}\chi_j{(y)}\chi_{k}(\xi)u(y) dyd\xi.$$
By the hypothesis on $A$ by Schur's lemma we have the uniform bound
$$\|A_{i,j,k}\|_{L^2\rightarrow L^2}\le C$$ 
for all $i,j,k$. 
The adjoint operators become
$$
A^*_{i',j',k'}u(y)=\int_{\mathbb{R}^{n}}\int_{\mathbb{R}^{n}}
e^{\ii(y-z)\cdot\eta}\overline{a(z,y,\eta)}\chi_{j'}(z)\chi_{i'}(y)\chi_{k'}(\eta)u(z) dzd\eta.
$$
Now composing  $A^*_{i',j',k'}\circ A_{i,j,k}$ we get
\begin{eqnarray*}
&& A_{i',j',k'}^*(A_{i,j,k}u)(x) \\
&=&\int_{\mathbb{R}^{2n}}\chi_{i'}{(y)}\chi_{j'}{(x)}\chi_{k'}(\xi)e^{\ii(x-y)\cdot\xi}\overline{a(y,x,\xi)}A_{i,j,k}u(y) dyd\xi\\
&=&\int_{\mathbb{R}^{4n}}\chi_{i'}{(y)}\chi_{j'}{(x)}\chi_{k'}(\xi)e^{\ii(x-y)\cdot\xi}\overline{a(y,x,\xi)}
\\& &\chi_{i}{(y)}\chi_{j}{(z)}\chi_{k}(\eta)e^{\ii(y-z)\cdot\eta}a(y,z,\eta)u(z)dzd\eta dyd\xi\\
&=&\int_{\mathbb{R}^{n}}u(z)dz\int_{\mathbb{R}^{3n}}\chi_{i'}{(y)}\chi_{j'}{(x)}\chi_{k'}(\xi)\chi_{i}{(y)}\chi_{j}{(z)}\chi_{k}(\eta)e^{\ii(x-y)\cdot\xi}e^{\ii(y-z)\cdot\eta}\\& &\overline{a(y,x,\xi)}a(y,z,\eta)dyd\xi d\eta
\\
&=&\int_{\mathbb{R}^{n}}K_{i,j,k,i',j',k'}(x,z)u(z)dz,
\end{eqnarray*}
where we denote
\begin{multline*}
K_{i,j,k,i',j',k'}(x,z) = \\
\int_{\mathbb{R}^{3n}}\chi_{i'}{(y)}\chi_{j'}{(x)}\chi_{k'}(\xi)\chi_{i}{(y)}\chi_{j}{(z)}\chi_{k}(\eta)
e^{\ii(x-y)\cdot\xi}e^{\ii(y-z)\cdot\eta}\overline{a(y,x,\xi)}a(y,z,\eta)dyd\xi d\eta.
\end{multline*} 
First we note that there is a constant $N_1$ such that for $|i-i'|> N_1$ we have
 $K_{i,j,k,i',j',k'}(x,z) =0$, so that we can restrict to $|i-i'|\le N_1.$
 
We can integrate inside $K$ withe the operator $L_y$ with the transpose
$$
{}^t L_y=\frac{1}{\ii} \frac{\eta-\xi}{|\eta-\xi|^2}\cdot\nabla_y,
$$
so that for $|k-k'|>N_2$, for some $N_2$, we have
\begin{eqnarray}\label{KERNELAMPLITUDINE}
&& K_{i,j,k,i',j',k'}(x,z) = \nonumber \\ 
&=&\int_{\mathbb{R}^{3n}}
e^{\ii(x-y)\cdot\xi}e^{\ii(y-z)\cdot\eta}
\chi_{i'}{(y)}\chi_{j'}{(x)}\chi_{k'}(\xi)\chi_{i}{(y)}\chi_{j}{(z)}\chi_{k}(\eta)
\overline{a(y,x,\xi)}a(y,z,\eta)dy d\xi d\eta \nonumber\\
&=&\int_{\mathbb{R}^{3n}}
e^{\ii(x-y)\cdot\xi}e^{\ii(y-z)\cdot\eta} \\
&& \quad
L_y^{2n+1}\left(\chi_{i'}{(y)}\chi_{j'}{(x)}\chi_{k'}(\xi)\chi_{i}{(y)}\chi_{j}{(z)}\chi_{k}(\eta)
\overline{a(y,x,\xi)}a(y,z,\eta)\right) dy d\xi d\eta \nonumber
\end{eqnarray}
From the assumptions, altogether, we obtain the estimate
\begin{equation}\label{EQ:estkk}
|K_{i,j,k,i',j',k'}(x,z)|\leq C\frac{1}{\langle k-k'\rangle^{2n+1}} 1_{|i-i'|\le N_1}.
\end{equation} 
Furthermore, let us define operators $L_\xi$ and $L_\eta$ by their transposes
$$
{}^t L_\xi=\frac{1}{\ii} \frac{x-y}{|x-y|^2}\cdot\nabla_\xi, \quad
{}^t L_\eta=\frac{1}{\ii} \frac{y-z}{|y-z|^2}\cdot\nabla_\eta.
$$
Then we can integrate by parts with these operators similar to and continuing \eqref{KERNELAMPLITUDINE}, so that
\begin{multline*}
K_{i,j,k,i',j',k'}(x,z) = 
\int_{\mathbb{R}^{3n}}
e^{\ii(x-y)\cdot\xi}e^{\ii(y-z)\cdot\eta} \\
L_\xi^{2n+1} L_\eta^{2n+1} L_y^{2n+1}\left(\chi_{i'}{(y)}\chi_{j'}{(x)}\chi_{k'}(\xi)\chi_{i}{(y)}\chi_{j}{(z)}\chi_{k}(\eta)
\overline{a(y,x,\xi)}a(y,z,\eta)\right) dy d\xi d\eta.
\end{multline*}
Combining this with \eqref{EQ:estkk}, we can estimate
\begin{equation}\label{EQ:estkkjs}
|K_{i,j,k,i',j',k'}(x,z)|\leq C\frac{1}{\langle k-k'\rangle^{2n+1}}
\frac{1}{\langle j-i'\rangle^{2n+1}}\frac{1}{\langle i-j'\rangle^{2n+1}} 1_{|i-i'|\le N_1}.
\end{equation} 
We can now estimate
$$
|i-j'|+|j-i'|\geq |i-i'+j-j'|\geq |j-j'|-|i-i'|\geq  |j-j'|-N_1.
$$
Consequently, we have
$$
 |j-j'|\leq |i-j'|+|j-i'|+N_1,
$$
and hence also
$$
\langle j-j'\rangle\leq C\langle i-j'\rangle \langle j-i'\rangle.
$$
Therefore, the estimate \eqref{EQ:estkkjs} implies
\begin{equation}\label{EQ:estkkjs2}
|K_{i,j,k,i',j',k'}(x,z)|\leq C\frac{1}{\langle k-k'\rangle^{2n+1}}
\frac{1}{\langle j-j'\rangle^{2n+1}} \frac{1}{\langle i-i'\rangle^{2n+1}}.
\end{equation} 
Since the support of $K_{i,j,k,i',j',k'}(x,z)$ is uniformly bounded in $x$ and $z$, and observing that the constants are quadratic in derivatives of the amplitude, we obtain
\begin{eqnarray*}
\sup_x \int_{\Rn} |K_{i,j,k,i',j',k'}(x,z)| dz & \leq & CM^2 h(i-i',j-j',k-k')^2, \\
\sup_z \int_{\Rn} |K_{i,j,k,i',j',k'}(x,z)| dx & \leq & CM^2 h(i-i',j-j',k-k')^2,
\end{eqnarray*}
where 
$$
M=\sup_{|\alpha|,|\beta|,|\gamma|\leq 2n+1}  \| \partial_x^{\alpha}\partial_{y}^{\beta}\partial_\xi^{\gamma} a\|_{L^\infty(\mathbb R^n_x\times \mathbb R^n_y \times \mathbb R^n_\xi)},
$$
and 
$$
h(i-i',j-j',k-k')=\left(\frac{1}{\langle k-k'\rangle^{2n+1}}
\frac{1}{\langle j-j'\rangle^{2n+1}} \frac{1}{\langle i-i'\rangle^{2n+1}}\right)^{\frac12}.
$$
Therefore, by Schur's lemma (see e.g. \cite[p. 284]{Stein} or \cite[Lemma 2.1]{RuzSug1}), we obtain 
$$
\|A_{i,j,k}^*A_{i',j',k'}\|_{L^2\rightarrow L^2}\le CM^2 h(i-i',j-j',k-k')^2.
$$
By a similar argument it can be also shown that 
$$
 \|A_{i,j,k}A_{i',j',k'}^*\|_{L^2\rightarrow L^2}\le CM^2 h(i-i',j-j',k-k')^2.
$$
Consequently, Cotlar's lemma (see e.g. \cite[Chapter VII, Section 2]{Stein} or \cite[Lemma 2.2]{RuzSug1}) implies that 
$$
\|A\|_{L^2\rightarrow L^2}\le CM,
$$
yielding the boundedness and \eqref{EQ:Cvampest}.
\endproof

As a consequence of Theorem \ref{THM:CV} we also get a version of the Calder\'on-Vaillancourt theorem for $\tau$-quantized operators.

\begin{corollary}
Let $\tau:\Rn\to\Rn$ be such that $\tau\in C^{4n+2}$, $\tau(0)=0$, and 
\begin{equation}\label{EQ:taubound}
\sup_{x\in\Rn} |\partial^{\alpha}\tau(x)|<\infty \; \textrm{ for all } \;
0<|\alpha|\leq 4n+2.
\end{equation} 
Let $\sigma:\Rn\times\Rn\to \mathbb C$ be such that 
\begin{equation}\label{bound}
\sup_{x,\xi\in\Rn} |\partial_{x}^{\beta}\partial_{\xi}^{\gamma}\sigma(x,\xi)|<\infty \; \textrm{ for all } \;  |\beta|, |\gamma|\le 2n+1.
\end{equation}
Then the operator
$$
A_{\sigma,{\tau}} u(x) = \Opt(\sigma) u(x) := \int_{\R^n} \int_{\R^n} e^{\ii (x - y) \cdot \xi} \sigma(x+\tau(y-x), \xi) u(y) \,d y \, d \xi
$$
is bounded
from $L^2(\Rn)$ to $L^2(\Rn)$
\end{corollary}
\proof
It is enough to use the chain rule applied to $a(x,y,\xi):=\sigma(x+\tau(y-x),\xi)$ and observe that one can apply Theorem \ref{THM:CV} since conditions
\eqref{EQ:taubound} and \eqref{bound} imply \eqref{EQ:CV-assump}.
\endproof

\section{Appendix: symmetry functions}
\label{SEC:heisenberg}

In this section we show how $\tau$-quantizations appear naturally in the analysis of operators in noncommutative settings. The trivial choice of $\tau={\rm id}$ in \eqref{bilfred} corresponding to the Kohn-Nirenberg type quantization has been extensively studied in the setting of general graded Lie groups in \cite{FR1}. 

Quantizations with general measurable functions $\tau$ 
in the form
\eqref{bilfred}
have been studied in
\cite{MR}. Especially, as explained in the introduction, for the exponential groups the symmetry functions always exist and lead to quantizations having the property of the Weyl quantization. Such symmetry functions may be given by the formula \eqref{EQ:Mrsym}, that is, by
\begin{equation}\label{EQ:Mrsym2}
\tau(x)=\int_{0}^{1}\exp[s\log x]ds.
\end{equation} 
Consequently, with this choice of $\tau$, the analogue of the `midpoint' $m(x,y)$ between points $x$ and $y$ is given by the formula 
\begin{equation}\label{EQ:midpoint}
m(x,y)=x\tau(y^{-1}x)^{-1}.
\end{equation} 

In this appendix we work out explicit examples of such symmetry functions in simple noncommutative settings of the polarised and of the normal Heisenberg groups. In the latter case we also calculate the explicit form of the  `midpoint' function $m(x,y)$.  

\subsection{Polarised Heisenberg group}

We will be using the matrix descriptions of the Heisenberg group and of its polarised version, see e.g. \cite[Section 6.1]{FR1} for more details on various descriptions of these groups.

Thus, we first consider the example of the polarised Heisenberg group $\mathbb{H}_{pol}$, which may be identified with the space of matrices 
$$
x=(a,b,c)\qquad\longleftrightarrow\qquad  \left[ {\begin{array}{ccc}
   1 & a & c \\
   0 & 1 & b \\
   0 & 0 & 1
  \end{array} } \right],
$$
with the Lie algebra $\mathfrak{H}_{pol}$ given by
$$
  \left[ {\begin{array}{ccc}
   0 & \tilde a & \tilde c \\
   0 & 0 & \tilde b \\
   0 & 0 & 0
  \end{array} } \right],
$$
where $a,b,c, \tilde a,\tilde b,\tilde c\in\mathbb{R}$.
Indeed, since the group $\mathbb{H}_{pol}$ is nilpotent (and hence also of exponential type), one can move from $\mathfrak{H}_{pol}$ to $\mathbb{H}_{pol}$ using the exponential function:
$$\exp
  \left[ {\begin{array}{ccc}
   0 & \tilde a & \tilde c \\
   0 & 0 & \tilde  b \\
   0 & 0 & 0
  \end{array} } \right]=\sum_{k=0}^{\infty}\frac{1}{k!}
  \left[ {\begin{array}{ccc}
   0 & \tilde a & \tilde c \\
   0 & 0 & \tilde b \\
   0 & 0 & 0
  \end{array} } \right]^k=
  \left[ {\begin{array}{ccc}
   1 & \tilde a & \tilde c +\frac{1}{2}\tilde a\tilde b\\
   0 & 1 & \tilde b \\
   0 & 0 & 1
  \end{array} } \right].
$$
It is also possible to do the reverse operation, in which case we have
$$\log
  \left[ {\begin{array}{ccc}
   1 & a & c \\
   0 & 1 & b \\
   0 & 0 & 1
     \end{array} } \right]
  =
  \left[ {\begin{array}{ccc}
   1 & a & c -\frac{1}{2}ab\\
   0 & 1 & b \\
   0 & 0 & 1
  \end{array} } \right].
$$
According to \eqref{EQ:Mrsym} a symmetry function on the polarised Heisenberg group $\mathbb{H}_{pol}$ can be given by the formula 
$$\tau(x)=\int_{0}^{1}\exp[s\log x]ds.$$
In particular, for $x=(a,b,c)$, we can calculate it explicitly:
\begin{eqnarray*}
\tau(x) & = & \int_{0}^{1}\exp\left[s\log \left[ {\begin{array}{ccc}
   1 & a & c \\
   0 & 1 & b \\
   0 & 0 & 1
     \end{array} } \right]\right]ds
     \\
     &=&\int_{0}^{1}\exp s \left[ {\begin{array}{ccc}
   0 & a & c-\frac{1}{2}ab \\
   0 & 0 & b \\
   0 & 0 & 0
     \end{array} } \right]ds\\
     &=&\int_{0}^{1}\exp\left[ {\begin{array}{ccc}
   0 & sa & sc-s\frac{1}{2}ab \\
   0 & 0 & sb \\
   0 & 0 & 0
     \end{array} } \right]ds\\
     &=&\int_{0}^{1}\left[ {\begin{array}{ccc}
   1 & sa & sc-s\frac{1}{2}ab+s^2\frac{1}{2}ab \\
   0 & 1 & sb \\
   0 & 0 & 1
     \end{array} } \right]ds\\
     &=&\left\{\left[ {\begin{array}{ccc}
   s & \frac{s^2}{2}a & \frac{s^2}{2}c-\frac{s^2}{4}ab+\frac{s^3}{6}ab \\
   0 & s & \frac{s^2}{2}b \\
   0 & 0 & s
     \end{array} } \right]\right\}_{0}^{1}\\
     &=&
     \left[ {\begin{array}{ccc}
   1 & \frac{1}{2}a & \frac{1}{2}c-\frac{1}{12}ab\\
   0 & 1 & \frac{1}{2}b \\
   0 & 0 & 1
     \end{array} } \right].
\end{eqnarray*}
Identifying the points of $\mathbb{H}_{pol}$ with points in $\mathbb{R}^3$, we get the formula
\begin{equation}\label{EQ:tau-heis}
\mathbb{H}_{pol}\ni x=(a,b,c)\mapsto \tau(x)=\left(\frac{a}{2},\frac{b}{2},\frac{c}{2}-\frac{ab}{12}\right).
\end{equation}  
According to the discussion in the introduction, it follows from \cite[Proposition 4.3]{MR} that the $\tau$-quantization with $\tau$ given by \eqref{EQ:tau-heis} would play the role of the Weyl quantization on the Heisenberg group in the sense that $\tau$-quantized operators with self-adjoint symbols would be also self-adjoint. Here we can note an interesting twist in the last variable on $\mathbb{H}_{pol}$ while it remains being the mid-point in the variables of the first stratum.

\subsection{Heisenberg group}

In a similar way we can do the same computations for the Heisenberg group which we will denote by $\mathbb{H}$, and its Lie algebra by $\mathfrak{H}$.
This group is given by the triples $(a,b,c)\in\mathbb{R}^{3}$ that verify the group law
$$
(a,b,c)\cdot(u,v,s)=(a+u,b+v,c+s+\frac{1}{2}(ub-va)).
$$
It is also possible to identify any element $(a,b,c)\in\mathbb{H}$ with an upper triangular matrix
via the association
$$x=(a,b,c)\qquad\longleftrightarrow\qquad \left[ {\begin{array}{ccc}
   1 & a & \frac{ab}{2}+c\\
   0 & 1 & b \\
   0 & 0 & 1
     \end{array} } \right].$$
     Similarly the elements of the Lie algebra of the Heisenberg group are given via the association
     $$(\tilde a,\tilde b,\tilde c)\qquad\longleftrightarrow\qquad \left[ {\begin{array}{ccc}
   0 & \tilde a & \tilde c\\
   0 & 0 & \tilde b\\
   0 & 0 & 0
     \end{array} } \right].$$
     As before it is possible to pass from the Lie algebra $\mathfrak{H}$ to the group $\mathbb{H}$ via the exponential mapping
     $$\exp\left[ {\begin{array}{ccc}
   0 & \tilde a & \tilde c\\
   0 & 0 & \tilde b \\
   0 & 0 & 0
     \end{array} }\right]=\sum_{k=0}^{\infty}\left[ {\begin{array}{ccc}
   0 & \tilde a & \tilde c\\
   0 & 0 & \tilde b \\
   0 & 0 & 0
     \end{array} }\right]^k=\left[ {\begin{array}{ccc}
   1 & \tilde a & \frac{\tilde a\tilde b}{2}+\tilde c\\
   0 & 1 & \tilde b \\
   0 & 0 & 1
     \end{array} }\right].$$
     Therefore, the logarithmic function is given by
     $$\log\left[ {\begin{array}{ccc}
   1 & a & \frac{ac}{2}+c\\
   0 & 1& b \\
   0 & 0 & 1
     \end{array} }\right]=\left[ {\begin{array}{ccc}
   0 & a & c\\
   0 & 0& b \\
   0 & 0 & 0
     \end{array} }\right].$$
     According to \eqref{EQ:Mrsym}, for $x=(a,b,c)$, a symmetry function on the Heisenberg group takes the form     \begin{eqnarray*}
     \int_{0}^{1}\exp(s\log x)ds&=&\int_{0}^{1}\exp\left(s\log \left[ {\begin{array}{ccc}
   1 & a & \frac{ab}{2}+c\\
   0 & 1& b \\
   0 & 0 & 1
     \end{array} }\right]\right)ds\\&=&\int_{0}^{1}\exp\left(s \left[ {\begin{array}{ccc}
   0 & a & c\\
   0 & 0& b \\
   0 & 0 & 0
     \end{array} }\right]\right)ds\\
     &=&\int_{0}^{1}\exp\left( \left[ {\begin{array}{ccc}
   0 & sa &sc\\
   0 & 0& sb \\
   0 & 0 & 0
     \end{array} }\right]\right)ds\\
     &=&\int_{0}^{1}\left[ {\begin{array}{ccc}
   1 & sa & \frac{s^2ab}{2}+sc\\
   0 & 1& sb \\
   0 & 0 & 1
     \end{array} }\right]ds
 \\
     &=&\left[ {\begin{array}{ccc}
   1 & \frac{a}{2} & \frac{ab}{6}+\frac{c}{2}\\
   0 & 1& \frac{b}{2} \\
   0 & 0 & 1
     \end{array} }\right].
       \end{eqnarray*}  
       Thus, a symmetry function can be given by the formula
\begin{equation}\label{EQ:tau-heis2}
\mathbb{H}\ni x=(a,b,c)\mapsto \tau(x)=\left(\frac{a}{2},\frac{b}{2},\frac{c}{2}+\frac{ab}{6}\right).
\end{equation}  
Furthermore, we have 
$$     
      \tau[(a,b,c)^{-1}]=\tau[(-a,-b,-c)]
      =\left(-\frac{a}{2},-\frac{b}{2},-\frac{c}{2}+\frac{ab}{6}\right).
$$
Thus, if $x=(a_1,b_1,c_1)$ and $y=(a_2,b_2,c_2)$, we get
\begin{eqnarray*}
\tau(y^{-1}x) &= &\tau((a_1-a_2,b_1-b_2,c_1-c_2+\frac12(b_1a_2-b_2a_1)))\\
& = & \left(\frac{a_1-a_2}{2},\frac{b_1-b_2}{2},\frac{c_1-c_2+\frac12(b_1a_2-b_2a_1)}{2}
+\frac{(a_1-a_2)(b_1-b_2)}{6}
\right) \\
& = &
\left(\frac{a_1-a_2}{2},\frac{b_1-b_2}{2},\frac{c_1-c_2}{2}+\frac{b_1a_2-b_2a_1}{4}
+\frac{(a_1-a_2)(b_1-b_2)}{6}\right).
\end{eqnarray*}
Consequently, we get the formula for the `midpoint' function
from \eqref{EQ:midpoint} as
\begin{eqnarray*}
m(x,y) & = & x\tau(y^{-1}x)^{-1} \\
&= &(a_1,b_1,c_1) \times \\
&&\; \times
\left(-\frac{a_1-a_2}{2},-\frac{b_1-b_2}{2},-\frac{c_1-c_2}{2}-\frac{b_1a_2-b_2a_1}{4}
-\frac{(a_1-a_2)(b_1-b_2)}{6}\right) \\
&=& \left(\frac{a_1+a_2}{2},\frac{b_1+b_2}{2},\frac{c_1+c_2}{2}
-\frac{(a_1-a_2)(b_1-b_2)}{6}
\right).
\end{eqnarray*}
In particular, we observe that if $a_1=a_2$ or if $b_1=b_2$, then
$$
m((a_1,b_1,c_1),(a_2,b_2,c_2))= \left(\frac{a_1+a_2}{2},\frac{b_1+b_2}{2},\frac{c_1+c_2}{2}
\right),
$$
but an additional quadratic twist appears in the central variable  if $a_1\not=a_2$ and $b_1\not=b_2$.


\begin{thebibliography}{10}
\bibitem{CV}
 A. Calder\'on, R. Vaillancourt, {\it On the boundedness of pseudo-differential operators}. J. Math. Soc. Japan, 23 (1971), 374-378.
 
\bibitem{Bayer}
 D. Bayer. {\em Bilinear time-frequency distributions and pseudodifferential operators}, Thesis, Universit\"at Wien, Wien, 2010.

 
\bibitem{CT}
 M. Cappiello, J. Toft, {\it Pseudo-differential operators in a Gelfand-Shilov setting}. Math. Nachr. 290 (2017), no. 5-6, 738-755. 
 
\bibitem{Constantine} G. M. Constantine, T. H. Savits, {\it A multivariate Faa di Bruno formula with applications}, Trans. Amer. Math. Soc. 348 (1996), no. 2, 503-520.

\bibitem{FR1} V. Fischer and M.  Ruzhansky, \emph{Quantization on nilpotent Lie groups}, Progress in Mathematics, Vol. 314, Birkh\"auser, 2016. (open access book) 
 
\bibitem{Lerner} N. Lerner, {\it Metrics on the phase space and non-selfadjoint pseudo-differetial operators}. Birkh\"auser, 2010.
\bibitem{Warren} P. Warren Johnson,  {\it The curious history of Fa\`a di Bruno formula}, Amer. Math. Monthly 109 (2002), no. 3, 217-234. 
\bibitem{MR} M. Mantoiu, M. Ruzhansky, {\it Pseudo-differential operators, Wigner transform and Weyl systems on type I locally compact groups}, Doc. Math., 22 (2017), 1539-1592.
\bibitem{RuzSug} M. Ruzhansky, M. Sugimoto, {\it On global inversion of homogeneous maps}, Bull. Math. Sci., 5 (2015), 13-18,
\bibitem{RuzSug1} M. Ruzhansky, M. Sugimoto, {\it Global $L^2$-boundedness theorems for a class of Fourier integral operators}, Comm. Partial Diff. Equations, 31:4, 547-569.
\bibitem{Ruz} M. Ruzhansky, V. Turunen {\it Pseudo-differential operators and symmetries}, Birkh\"auser, 2010.
\bibitem{Stein} E. M. Stein, {\it Harmonic Analysis: Real-Variable Methods, Orthogonality, and Oscillatory Integrals}. Princeton: Princeton University Press, 1993. 

\bibitem{Schubin}  M. A. Shubin, {\it Pseudodifferential operators and spectral theory}, Springer, 2001.

\bibitem{Taylor} M. E. Taylor, {\em Partial differential equations II. Qualitative studies of linear equations}. Second edition. Applied Mathematical Sciences, 116. Springer, New York, 2011.

\bibitem{Toft} J. Toft, {\it Matrix parameterized pseudo-differential calculi on modulation spaces.} Generalized functions and Fourier analysis, 215-235, Oper. Theory Adv. Appl., 260, Adv. Partial Differ. Equ. (Basel), Birkh\"auser/Springer, Cham, 2017.

\bibitem{Wong} M. W. Wong {\it An introduction to pseudo-differential operators}, World Scientific, 2014.
\end{thebibliography}
\end{document}